\documentclass[12pt,reqno]{amsart}
\usepackage[utf8]{inputenc}
\usepackage{fullpage}

\usepackage{latexsym,amsfonts,amssymb,amsmath,amsthm}
\usepackage{graphicx}

\usepackage[usenames,dvipsnames]{color}
\usepackage{ulem}
\usepackage{comment}

\usepackage{url}
\let\oldtocsection=\tocsection

\let\oldtocsubsection=\tocsubsection

\renewcommand{\tocsection}[2]{\hspace{0em}\oldtocsection{#1}{#2}}
\renewcommand{\tocsubsection}[2]{\hspace{1em}\oldtocsubsection{#1}{#2}}

\usepackage{mathrsfs}
\usepackage{enumerate}
\usepackage{booktabs} 
\usepackage{pgfplots}
\usepackage{soul}
\usepackage[all]{nowidow}
\usepackage[utf8]{inputenc}
\usepackage{graphics,graphicx}
\usepackage{tikz}
\usetikzlibrary{automata}
\usepackage{multicol}
\usepackage{algpseudocode,algorithm,algorithmicx}
\usepackage{multirow}
\usepackage{multicol}
\usepackage[numbers,sort&compress]{natbib}

\usepackage{color}

\usepackage{graphics}
\usepackage{caption}
\usepackage{subcaption}

\usepackage{bbm}
\usepackage{enumerate}
\usepackage{amsmath,bm}

\definecolor{blue}{HTML}{1F77B4}
\definecolor{orange}{HTML}{FF7F0E}
\definecolor{green}{HTML}{2CA02C}

\newtheorem{tm}{Theorem}[section]
\newtheorem{prop}{Proposition}[section]

\newtheorem{lem}{Lemma}[section]

\newtheorem{rk}{Remark}[section]

\numberwithin{equation}{section}
\numberwithin{tm}{section}

\def\p{\partial}
\def\R{{\mathbb R}}

\def\be{\beta}

\def\De{\Delta}

\def\Om{\Omega}

\usepackage{fancyhdr}

\begin{document}

\pagestyle{fancy}
\lhead{}
\chead{}
\rhead{}
\lfoot{}
\cfoot{\thepage}
\rfoot{}
\renewcommand{\headrulewidth}{0pt}
\setlength{\footskip}{25pt}

\title[Spatial profiles of  a diffusive epidemic model]
{Spatial profiles of a reaction-diffusion epidemic model with nonlinear incidence mechanism and varying total population}
\author[R. Peng, R.B. Salako and Y. Wu]{Rui Peng, Rachidi B. Salako and Yixiang Wu}

\thanks{{\bf R. Peng}: School of Mathematical Sciences, Zhejiang Normal
University, Jinhua, 321004, Zhejiang, China. {\bf Email}: {\tt pengrui\_seu@163.com}}

\thanks{{\bf R. B. Salako}: Department of Mathematical Sciences, University of Nevada Las Vegas, Las Vegas, NV 89154, USA. {\bf Email}: {\tt rachidi.salako@unlv.edu}}

\thanks{{\bf Y. Wu}: Department of Mathematical Sciences, Middle Tennessee State University, Murfreesboro, TN 37132, USA. {\bf Email}: {\tt yixiang.wu@mtsu.edu}}

\date{\today}



\begin{abstract} This paper considers a susceptible-infected-susceptible (SIS) epidemic reaction-diffusion model with no-flux boundary conditions and varying total population. The interaction of the susceptible and infected people is describe by the nonlinear transmission mechanism of the form $S^qI^p$, where $0<p\le 1$ and $q>0$. In \cite{PSW}, we have studied  a model with a constant total population. In the current paper, we extend our analysis to a model with a varying total population, incorporating birth and death rates. We investigate the asymptotic profiles of the endemic equilibrium when the dispersal rates of susceptible and/or infected individuals are small. Our work is motivated by disease control strategies that limit population movement. To illustrate the main findings, we conduct numerical simulations and provide a discussion  of the theoretical results from the view of disease control. We will also compare the results for the models with constant or varying total population.

\end{abstract}

\maketitle


\section{Introduction}

Ordinary differential equation compartmental epidemic models have been widely adopted to study the transmission of infectious diseases since the pioneering work of Kermack and McKendrick \cite{KM1}. In these models, the population is typically divided into distinct groups, such as susceptible and infected individuals, and the transmission dynamics between these groups are described by a system of differential equations. To account for the impact of spatial heterogeneity and population movement on the spread of infectious diseases, recent efforts have delved into the exploration of reaction-diffusion epidemic models with non-constant coefficients (e.g., \cite{allen2008asymptotic,Bertaglia,CL,CLL,DP,KMP,LXZ,LPW,LPX,LT2023,LLou,LS,MWW2,Peng,Sa,SLX,WWK}).
These models offer a potential tool for understanding the underlying mechanisms and predicting the spread of infectious diseases in spatially heterogeneous environments.

In differential equation epidemic models, the interaction between susceptible and infected individuals is typically described by a transmission term, which is crucial for the accuracy of these models. A commonly used transmission mechanism is the mass action term $\beta SI$, adopted in the pioneering work of Kermack and McKendrick \cite{KM1}. This mechanism assumes that the number of new infections is directly proportional to the densities of susceptible ($S$) and infected ($I$) individuals, with $\beta$ representing the disease transmission rate. However, the mass action mechanism may not accurately represent disease transmission in many scenarios \cite{DV1,Het1,McCallum}. An alternative is the nonlinear incidence mechanism $\beta S^qI^p$, where $p, q > 0$, which offers greater flexibility and may better capture the dynamics of disease spread \cite{Het2,HLV,HV,LHL,LLI,LR}.


In \cite{PSW}, we have considered the following susceptible-infected-susceptible (SIS) reaction-diffusion epidemic model:
\begin{equation}\label{model-1}
    \begin{cases}
        \p_tS=d_S\Delta S-\beta(x) S^qI^p+\gamma(x) I,\ \ \ & x\in\Om,  t>0,\cr
        \p_t I= d_I\Delta I +\beta(x) S^qI^p-\gamma(x) I, & x\in \Om,\ t>0,\cr
        \p_{\nu}S=\p_{\nu}I=0, & x\in\p \Om, \ t>0,\cr
        S(x,0)=S_0(x),\ \ I(x,0)=I_0(x),&x\in\Omega,\cr
        \int_\Omega (S+I)=N, &t>0,
    \end{cases}
\end{equation}
where $S(x,t)$ and $I(x,t)$ are  the density of  susceptible and infected individuals at  spatial location $x$ and time $t$, respectively. The positive constants $d_S$ and $d_I$ are population movement rates; $\Omega\subset \R^n$ is a bounded domain with smooth boundary $\p\Om$; $\nu$ is the outward unit normal to $\p\Om$; and $\beta$ and $\gamma$ are disease transmission and recovery rates, respectively, which are  positive H\"older continuous functions on $\bar{\Om}$.
The nonlinear term $\beta S^qI^p$ describes the interaction of the susceptible and infected population. The homogeneous Neumann boundary condition means that
the individuals cannot cross the boundary $\partial\Omega$. Integrating the sum of the first two equations in \eqref{model-1}, we obtain that the total population $N$ remains a positive constant. In the past decade, this system has been studied extensively, for instance, in \cite{castellano2022effect,Castellano2023multiplicity,CLPZ,deng2016dynamics,deng2023dynamics-corr,li2020asymptotic,PW2021,WJL2018,wu2016asymptotic}.



In a more realistic setting, the total population  is not conserved as individuals may be born, die, or emigrate. Taking into account population recruitment/immigration and disease induced mortality,  we obtain the following modified model that will be investigated in the current paper:
\begin{equation}\label{model-2}
\begin{cases}
\p_tS=d_S\De S+\Lambda(x)-S-\beta(x) S^qI^p+\gamma(x) I,\ \ \ & x\in\Omega, t>0,\cr
\p_tI=d_I\De I +\beta(x) S^qI^p-(\gamma(x)+\eta(x)) I, & x\in\Omega, t>0,\cr
\p_{\nu}S=\partial_{\nu}I=0, & x\in\p\Omega, t>0, \cr
 S(x,0)=S_0(x),\ \ I(x,0)=I_0(x),&x\in\Omega.
\end{cases}
\end{equation}
In \eqref{model-2}, $\Lambda-S$ is  the recruitment term, representing that the susceptible population is subject to a linear growth (the coefficient before $S$ is assumed to be a positive constant and normalized to one for simplicity of presentation); $\eta$ stands for disease induced mortality rate; $\beta$ and $\gamma$  have the same meaning as in \eqref{model-1}. All coefficient functions  are positive H\"older continuous functions on $\bar{\Om}$. {  A more general recruitment term, $\Lambda - \theta(x) S$, where $\theta$ is a positive continuous function on $\bar{\Omega}$, can also be considered. All the results in this paper remain valid with minor modifications. For simplicity of presentation, we use the term $\Lambda - S$ here.} For detailed biological interpretations of \eqref{model-2}, one may refer to \cite{Het1,Het2,li2018diffusive}.

An equilibrium of \eqref{model-2} is a classical solution of the following elliptic system:
\begin{equation}\label{1-2}
\begin{cases}
    d_S\De S+\Lambda(x)-S-\beta(x) S^qI^p+\gamma(x) I=0,\ \ \ & x\in\Omega,\cr
    d_I\De I+\beta(x) S^qI^p-(\gamma(x)+\eta(x)) I=0, & x\in\Omega,\cr
        \p_{\nu}S=\partial_{\nu}I=0, & x\in\p\Omega.
    \end{cases}
\end{equation}
We call a nonnegative equilibrium $(S, I)$  a \textit{disease free equilibrium} (DFE) if $I=0$ and an \textit{endemic equilibrium} (EE) if $I\neq 0$. By the maximum principle, an  EE $(S,I)$ of \eqref{model-2}
satisfies $S,\,I>0$ on $\bar\Omega$.

In this paper, we aim to study the asymptotic profiles of the EEs of \eqref{model-2} as the dispersal rates $d_S$ and/or $d_I$ approach zero. Such results may reflect the consequences of disease control strategies of limiting population movement. In  \cite{PSW}, we 
studied the asymptotic profiles of the EEs of \eqref{model-1}. As in \cite{PSW}, we will focus on the case $0<p\le 1$ and $q>0$, as the model may have  bistable dynamics when $p>1$ \cite{FCGT,PW2021}.  { Biologically, the parameter range $0<p<1$ in $I^p$ 
  represents a scenario where the infection rate increases faster than linearly when the number of infected individuals is low, but increases more slowly than linearly as the number of infected individuals becomes high \cite{LHL}.} For system \eqref{1-2}, previous studies in \cite{li2018diffusive,PWZZ2023} have explored the case of $q=p=1$. Specifically, \cite{li2018diffusive} investigated the scenario where $d_S\to 0$, while both \cite{li2018diffusive} and \cite{PWZZ2023} considered the case of $d_I\to 0$ in a one-dimensional domain $\Omega$. The obtained results in this paper significantly improve or extend those in these earlier works. { We remark that another approach to investigate the impact of population movement constraint on disease spread is to study  the time-dependent systems (e.g. \eqref{model-1} and \eqref{model-2})  with $d_S\approx 0$ or $d_I\approx 0$. We refer interested readers to \cite{Burie1, Burie2, Sa1} for research in this direction.}

We would like to stress that the techniques employed in the analysis of \eqref{model-1} or those utilized in \cite{li2018diffusive,PWZZ2023} for \eqref{1-2} are largely inapplicable to system \eqref{1-2}. For instance, when $p=1$, the equilibrium problem associated with \eqref{model-1} can be reduced to studying certain problems involving an algebraic equation and an elliptic equation. However, the equilibrium problem for system \eqref{1-2} requires the investigation of solutions for a two-species cooperative system, a two-species predator-prey system, or other distinct scenarios. Hence, novel approaches need to be developed to analyze the spatial profiles of solutions to system \eqref{1-2}, particularly in cases involving higher spatial dimensions and the general nonlinear infection mechanism $\beta S^qI^p$.



This paper is  structured as follows. In Section \ref{Main-Results-section}, we state the main results obtained in the paper. The proofs of the main results are given in Sections 3 and 4, where the cases $p=1$ and $0<p<1$ are considered respectively.
In Section \ref{Nuremical-Sim-section}, we present numerical simulations specifically for the case  $p=1$ to support and complement the theoretical results. The discussions and comparison of the results for both models \eqref{model-1} and \eqref{model-2} are also presented in Section \ref{Nuremical-Sim-section}.
The appendix in Section \ref{Appendix-section} includes supplementary results that are utilized in our proofs.

\section{Main results}\label{Main-Results-section}
For system \eqref{model-2}, denote by $(\tilde{S},0)$ the unique DFE, where $\tilde{S}$ is the unique positive solution of
\begin{equation}\label{DFE_mass_bd}
d_S\Delta \tilde{S}-\tilde{S}+\Lambda=0,\ \  x\in\Om;\ \ \
\partial_{\nu}\tilde{S}=0, \ \ x\in\partial\Om.
\end{equation}
 If $p=1$, similar to \cite{allen2008asymptotic,li2018diffusive}, the basic reproduction number of \eqref{model-2} is defined by
\begin{equation}\label{R_0-q-2}
\mathcal{R}_{0}=   \sup_{\varphi\in W^{1,2}(\Om)\setminus\{0\}}\frac{\int_{\Om}\be\tilde S^q\varphi^2}{\int_{\Om}(d_I|\nabla \varphi|^2+(\gamma+\eta)\varphi^2)}.
 \end{equation}
Notice that $\mathcal{R}_{0}$ depends on both  $d_S$ and $d_I$.

We first discuss about the existence, uniqueness and uniform persistence of the solutions, and the existence of EEs for  model  \eqref{model-2}. For our purpose, as in \cite{PW2021}, we need to impose the following assumption on the initial data:
\begin{enumerate}
\item[{\rm(A)}] $S_0$ and $I_0$ are nonnegative continuous functions on $\bar\Omega$. Moreover,
\begin{enumerate}
\item[(i)] The initial value $I_0\geq,\not\equiv0$ on $\bar\Omega$;
\item[(ii)] If $0<q< 1$, $S_0(x)>0$ for all $x\in\bar\Omega$;
\item[(iii)] If $0<p<1$, $I_0(x)>0$ for all $x\in\bar\Omega$.
\end{enumerate}
\end{enumerate}


The proof of the following result is similar to that of  \cite[Theorem 2.1]{PW2021}, and we only sketch it in the appendix.
\begin{prop}\label{theorem_ex} Suppose that $0<p\le 1, q>0$ and {\rm (A)} holds. Then the following statements hold.
\begin{enumerate}
\item[{\rm(i)}] System \eqref{model-2} has a unique global classical solution $(S, I)$ with $S(x, t), I(x, t)>0$ for all $x\in\bar\Omega$ and $t>0$, and there exists $M_\infty>0$ depending on the initial data such that
\begin{equation}\label{th-b1}
\|S(\cdot, t)\|_{L^\infty(\Omega)}, \ \ \|I(\cdot, t)\|_{L^\infty(\Omega)} \le M_\infty, \ \  t\ge 0.
\end{equation}
Moreover, there exists $N_\infty>0$ independent of initial data such that
\begin{equation}\label{th-b2}
\max\{\limsup_{t\rightarrow\infty}\|S(\cdot, t)\|_{L^\infty(\Omega)},  \ \limsup_{t\rightarrow\infty}\|I(\cdot, t)\|_{L^\infty(\Omega)}\} \le N_\infty.
\end{equation}

\item[{\rm(ii)}] { For any $0<p<1$ or $p=1$ and $\mathcal{R}_{0}>1$,}
 there exists $\epsilon_0>0$ independent of initial data such that for any solution $(S, I)$ of \eqref{model-2}, we have
\begin{equation}\label{persist}
\min\{\liminf_{t\rightarrow\infty} S(x, t), \  \liminf_{t\rightarrow\infty}I(x, t)\} \ge \epsilon_0,
\end{equation}
uniformly for $x\in\bar\Omega$. Moreover, \eqref{model-2} has at least one EE.
\end{enumerate}
\end{prop}

Proposition \ref{theorem_ex} asserts that \eqref{model-2} possesses at least one EE if either $0<p<1$ or $p=1$ and $\mathcal{R}_{0}>1$. In the subsequent part of this section, we present the main findings regarding the asymptotic profiles of the EEs as the dispersal rates $d_S$ and/or $d_I$ tend to zero. Whenever an EE solution of \eqref{model-2} exists, we denote it as $(S, I)$.



For $f\in C(\bar\Omega)$, let
 $$
 f_{\text{min}}=\min_{x\in\bar\Omega}f(x)\ \ \ \text{and} \ \ f_{\text{max}}=\max_{x\in\bar\Omega}f(x).
 $$

Define
$$
h(x)=\frac{\gamma(x)+\eta(x)}{\beta(x)},\ \ \ \ \ x\in\bar\Omega,
$$
which is called the risk function for system \eqref{model-2}  (if $d_S=d_I=0$ and $p=1$, then $\Lambda^q/h$ is the basic reproduction number of \eqref{model-2}). { For convenience, denote 
$$
r(x)=\frac{\gamma(x)}{\beta(x)},\ \ \ \ \ x\in\bar\Omega.
$$}

\subsection{Asymptotic profile of EE as $d_I\to 0$}

We first consider the case $d_I\to 0$, and have the following result.
\begin{tm}\label{TH2.8}
Fix $d_S,\ q>0$. The following statements hold.
\begin{enumerate}
\item[\rm(i)] Suppose that $p=1$. {  Let $\tilde{\Omega}:=\{x\in\bar\Om:\ \tilde S(x)>h^{\frac{1}{q}}(x)\}.$ Then  \eqref{model-2} has an EE $(S, I)$  when $0<d_I\ll 1$ if and only $\tilde{\Omega}\neq\emptyset$. Suppose $\tilde\Omega\neq \emptyset$.} Then up to a subsequence if necessary, $S\to S^*$ wealky in $W^{1,2}(\Omega)$ and weakly-star in $L^{\infty}(\Omega)$, and $I\to \mu$ weakly-star in $[C(\bar\Omega)]^*$ as $d_I\to 0$, where $S^*\in L^{\infty}(\Omega)\cap W^{1,2}(\Omega)$ and $\mu$ is a finite Radon measure on $\bar\Omega$. Moreover, $\mu(\bar\Omega)>0$,
\begin{equation} \label{D6-2}
  \lim_{d_I\to 0}\|I\|_{L^{\infty}(K)}=0
\end{equation}
for every compact set $K\subset \big\{x\in\bar\Omega:\ \tilde{S}(x)<h^{\frac{1}{q}}(x)\big\}$,
\begin{equation}
  d_S\int_{\Omega}\nabla S^*\cdot \nabla \varphi -\int_{\Omega}S^*\varphi+\int_{\Omega}\Lambda \varphi -\int_{\Omega}\eta\varphi d\mu=0,   \label{D6-3}
\end{equation}
for all $\varphi\in W^{1,2}(\Omega)\cap C(\bar\Omega)$,
and
\begin{equation}
   m\le S^*\le h^{\frac{1}{q}}  \label{D6-3-3}
\end{equation}
for some constant $m>0$,
and there exist a measurable set $F_1$ with Lebesgue measure zero and closed sets $F_l$, $l\ge 2$, with $\Omega=\cup_{l\ge 1} F_l$ such that $S^*$ is continuous on $F_l$ for all $l\ge 2$ and
$$
\mu(\{x\in\cup_{l\ge 2}F_l:\ S^*(x)\neq h^{\frac{1}{q}}(x)\})=0
$$
and
$$
\mu(\{x\in\cup_{l\ge 2}F_l:\ S^*(x)= h^{\frac{1}{q}}(x)\}\cup F_1)>0.
$$

\item[\rm(ii)] Suppose that $0<p<1$. Then \eqref{model-2} admits an EE for any $d_I>0$, and  $\left(S,I\right)\rightarrow \left(S_*, I_*\right)$ uniformly on $\bar{\Omega}$ as $d_I\rightarrow0$,
where $I_*>0$ on $\bar\Omega$ is given by
 \begin{equation}
 \label{S-lim-1}
I_*=\left(\frac{S_*^q}{h}\right)^{\frac{1}{1-p}},
 \end{equation}
and $S_*$ is the unique positive solution of
 \begin{equation}
 \label{S-lim-2}
 d_S\Delta  S_*+\Lambda-S_*-\eta \left(\frac{S_*^q}{h}\right)^{\frac{1}{1-p}}=0,\ x\in \Omega;\ \
 \p_{\nu}S_*=0,\ x\in \partial\Omega.
 \end{equation}

\end{enumerate}
\end{tm}

\begin{rk}\label{r-th2.8} Theorem \ref{TH2.8}(i) implies that $S^*$ satisfies $-d_S\Delta S^*=\Lambda-S^*$ in the region $\{x\in\Omega:\ S^*(x)\neq h^{\frac{1}{q}}(x)\}$, and the expression for $\mu$ is given by
$$
\mu(\{x\})=\frac{d_S\Delta (h^{\frac{1}{q}}(x))+\Lambda(x)-h^{\frac{1}{q}}(x)}{\eta(x)},\ \ \ \forall\, x\in(\cup_{l\ge 2}F_l)\cap\{S^*= h^{\frac{1}{q}}\},
$$
when $h\in C^2(\bar\Omega)$.
It is important to note that the distribution of $\mu$ heavily relies on the risk function $h$, as demonstrated in the one-dimensional case by \cite{PWZZ2023}. In fact, for a large class of risk functions $h$ when the habitat $\Omega$ is a finite open interval, \cite[Theorem 2.3]{PWZZ2023} provides a precise description of the disease distribution. It is worth mentioning that the techniques employed in the proof of \cite[Theorem 2.3]{PWZZ2023} are applicable to the general case of $p=1$ and $q>0$, but only for the one-dimensional domain. On the other hand, Theorem \ref{TH2.8}(ii) appears to be new in the existing literature.

\end{rk}

\subsection{Asymptotic profile of EE as $d_S\to 0$}
Next, we consider the case $d_S\to 0$. As a preparation, we first note that a standard singular perturbation argument for elliptic equations shows that $\tilde S$ converges uniformly to $\Lambda$ as $d_S\to 0$ (see, e.g. \cite[Lemma 3.2]{peng2008stationary} or \cite[Lemma 2.4]{du2009effect}). Similar to the analysis in \cite[Subsection 3.2]{li2018diffusive}, the existence of an EE of \eqref{model-2} for small $d_S$ can be ensured by imposing $\lambda_0<0$, where $\lambda_0$ is the principal eigenvalue of the following eigenvalue problem.
 \begin{equation}
 d_I\Delta \varphi+\left(\beta \Lambda^q-\gamma-\eta\right)\varphi+\lambda \varphi=0,\ x\in\Omega;\ \ \
 \p_{\nu}\varphi=0,\ x\in\partial\Omega.
 \label{eigen-prob2}
 \end{equation}

As $d_S\to 0$, we have the following result about the asymptotic behavior of the EE of \eqref{model-2}.

\begin{tm}\label{TH2.9} Fix $d_I,\ q>0$. Then the following statements hold.

\begin{enumerate}
\item[\rm(i)] Suppose that $p=1$. Assume that $\lambda_0<0$ such that \eqref{model-2} has at least one EE for $0<d_S\ll 1$. Then up to a subsequence if necessary,  $\left(S,I\right)\rightarrow \left(S_*, I_*\right)$ uniformly on $\bar{\Omega}$ as $d_S\rightarrow0$, where $S_*>0$ on $\bar\Omega$ fulfills
$\Lambda-S_*-\beta S_*^q I_*+\gamma I_*=0$ and $I_*$ is a positive solution of
 \begin{equation}
 \nonumber
 -d_I\Delta  I_*=\beta S_*^q  I_*-(\gamma+\eta)  I_*,\ x\in \Omega;\ \
 \p_{\nu} I_*=0,\ x\in \partial\Omega.
 \end{equation}

\item[\rm(ii)] Suppose that $0<p<1$. Then \eqref{model-2} has at least one EE for any $d_S>0$, and up to a subsequence if necessary, $\left(S,I\right)\rightarrow \left(S_*, I_*\right)$ uniformly on $\bar{\Omega}$ as $d_S\rightarrow0$,  where $S_*>0$ on $\bar\Omega$ satisfies
$$
\Lambda- S_*-\beta S_*^q I_*^p+\gamma I_*=0
$$
and $I_*$ is a positive solution of
 \begin{equation}
 \label{I-lim}
 -d_I\Delta   I_*=\beta S_*^q  I_*^p-(\gamma+\eta) I_*,\ x\in \Omega;\ \
 \p_{\nu} I_*=0,\ x\in \partial\Omega.
 \end{equation}

\end{enumerate}

\end{tm}

\begin{rk}\label{r-th2.9a} Theorem \ref{TH2.9}(i) generalizes \cite[Theorem 3.1]{li2018diffusive}, which specifically addresses the case $p=q=1$, and Theorem \ref{TH2.9}(ii) seem to be new in the literature.
\end{rk}

 \subsection{Asymptotic profile of EE as $d_S,\,d_I\to 0$}
 The following results state the asymptotic behavior of the EE of \eqref{model-2} as $d_S$ and $d_I$ approach zero. We have to impose the condition $\{x\in\bar\Om:\ \Lambda(x)>h^{\frac{1}{q}}(x)\}\ne\emptyset$ such that  system  \eqref{model-2} admits at least one EE when both movement rates are sufficiently small.

\begin{tm}\label{TH2.10} Fix $q, \sigma>0$. Then the following statements hold.

\begin{enumerate}
\item[\rm(i)] Suppose that $p=1$. Let $\{x\in\bar\Om:\ \Lambda(x)>h^{\frac{1}{q}}(x)\}\neq\emptyset$ such that  \eqref{model-2} has at least one EE  when $0<d_S, d_I\ll 1$. Suppose $\tilde\Omega\neq \emptyset$.
Then up to a subsequence if necessary, $(S,I)\to (S^*,I^*)$ weakly-star in $L^{\infty}(\Omega)$ as $d_{I}+|\frac{d_{I}}{d_{S}}-\sigma|\to 0$, where $S^*,\,I^*\in L^{\infty}(\Omega)$ satisfy
    \begin{eqnarray}
        &&\Lambda-S^*-\eta I^*=0\ \ \ \text{a.e in}\ \Om,  \label{TH9-eq-1}\\
       && \min\{\Lambda_{\min},r_{\min}^{\frac{1}{q}}\}\le S^*\le \Lambda_{\max} \ \ \ \text{a.e. in}\ \Om, \label{TH9-eq-2}\\
       && 0\le I^*\le \frac{1}{\min\{\sigma, {\eta}\}}(\Lambda-h^{\frac{1}{q}})_+ \ \   \ \text{a.e. in}\ \Om \ \text{and}\ \int_{\Om}I^*>0, \label{TH9-eq-3}\\
        &&\frac{1}{\eta}(\Lambda-h)_+\le I^* \ \text{and}\ S^*\le \min\{\Lambda,\ h\}\ \ \text{a.e. in}\ \Om\  \text{if}\ q=1.\label{TH9-eq-4}
   \end{eqnarray}
In addition if  $\sigma\geq\eta_{\max}$,  we have
 \begin{equation}\nonumber
(S,I)\to\left(\min\{\Lambda,\ h^{\frac{1}{q}}\},\ \frac{1}{\eta}\big(\Lambda-h^{\frac{1}{q}}\big)_+\right)
 \end{equation}
uniformly on $\bar\Om$ as $d_I+|\frac{d_I}{d_S}-\sigma|\to 0$.

\item[\rm(ii)] Suppose that $0<p<1$. Then \eqref{model-2} has at least one EE for any $d_S, d_I>0$.  Up to a subsequence if necessary, $(S,I)\to (S^*,I^*)$ weakly-star in $L^{\infty}(\Omega)$ as $d_{I}+|\frac{d_{I}}{d_{S}}-\sigma|\to 0$, where $S^*,\,I^*\in L^{\infty}(\Omega)$ and satisfy
 \begin{eqnarray}
        &&S^*+\eta I^*=\Lambda\ \ \ \ \ \ \ \ \ \ \text{a.e.\ in}\ \Om,  \label{model-di-ds-eq-1}\\
       && c_*\le S^*\le \Lambda_{\max}\ \ \quad \   \ \text{a.e.\ in}\ \Om, \label{model-di-ds-eq-2}\\
       && c_*\le I^*\le \frac{\Lambda_{\max}}{\eta_{\min}} \ \ \ \quad  \ \text{a.e.\ in}\ \Om \label{model-di-ds-eq-3}
   \end{eqnarray}
   for some constant $c_*>0$. In addition if  $\sigma>\eta_{\max}$, then $(S,I)\to(S^*,I^*)$ uniformly on $\bar\Om$ as $d_I+|\frac{d_I}{d_S}-\sigma|\to 0$, where $S^*=h^{\frac{1}{q}}\left(I^*\right)^{\frac{1-p}{q}}$ and $I^*$ is the unique positive solution of
   \begin{equation}
       \Lambda-\eta I^*-h^{\frac{1}{q}}\left(I^*\right)^{\frac{1-p}{q}}=0.
   \end{equation}

\end{enumerate}

\end{tm}


\begin{rk}\label{r-th2.9} Theorem \ref{TH2.10} appears to be a new result, even in the special case of $p=q=1$ for Theorem \ref{TH2.10} (i), which corresponds to the mass action infection function, or in the context of a one-dimensional domain.
\end{rk}



\vskip10pt


\section{Proofs of the main results: case of $p=1$} In this section, we assume that $p=1$ and establish Theorem \ref{TH2.8}(i) and Theorem \ref{TH2.10}(i). Since the proof of Theorem \ref{TH2.9}(i) is similar to that of \cite[Theorem 3.1]{li2018diffusive}, we  omit it here.


\subsection{Some useful lemmas} The following lemma plays a vital role in our later analysis.
 \begin{lem}[{\cite[Lemma 2.2]{jiang2007priori}, \cite[Lemma 3.1]{peng2013qualitative}}]\label{l3.1} Suppose that $w\in C^2({\bar\Omega})$
and $\partial_{\nu}w=0$ on $\partial\Omega$. The following
statements hold.

{\rm (i)} If $w$ attains a local maximum at $x_1\in\bar\Omega$,
then $\nabla w(x_1)=0$ and $\Delta w(x_1)\leq0$.

{\rm (ii)} If $w$ attains a local minimum at $x_2\in\bar\Omega$,
then $\nabla w(x_2)=0$ and $\Delta w(x_2)\geq0$.
 \end{lem}

\vskip6pt
We can apply Lemma \ref{l3.1} to establish positive upper and lower bounds for the S-component of the EEs.
\begin{lem}\label{lemma_Sbound} Let $(S, I)$ be an EE of \eqref{1-2} with $p=1$. Then it holds that
    $$
    \min\left\{\Lambda_{\min},\ \, r_{\min}^{\frac{1}{q}}\right\}\le S_{\min}\le S_{\max}\le \max\left\{\Lambda_{\max},\ \, r_{\max}^{\frac{1}{q}}\right\},\ \ \ \forall d_S,\, d_I>0.
    $$
\end{lem}
\begin{proof} Suppose that $S_{\max}=S(x_0)$ for some $x_0\in\bar\Omega$. Then by Lemma \ref{l3.1}(i), we have
$$
\Lambda(x_0)-S(x_0)-\beta(x_0)S^q(x_0)I(x_0)+\gamma(x_0)I(x_0)\ge 0,
$$
which implies
$$ S(x_0)+S^q(x_0)\beta(x_0)I(x_0)\le \Lambda(x_0)+\gamma(x_0)I(x_0)\le \Lambda_{\max}+r_{\max}\beta(x_0)I(x_0).$$
Hence,
\begin{equation}\label{lamS}
(S^q(x_0)-r_{\max})\beta(x_0)I(x_0)\le \Lambda_{\max}-S(x_0).
\end{equation}
If $\Lambda_{\max}\ge S(x_0)$, then we have
$$
S_{\max}\le \max\big\{\Lambda_{\max}, \ r_{\max}^{\frac{1}{q}}\big\}.
$$
On the other hand, if $ \Lambda_{\max}<S(x_0)$, we  obtain from \eqref{lamS} that $$
S(x_0)<r_{\max}^{\frac{1}{q}}\le \max\big\{\Lambda_{\max},\ r_{\max}^{\frac{1}{q}}\big\}.
$$ Consequently, $S_{\max}\le \max\big\{\Lambda_{\max},\ r_{\max}^{\frac{1}{q}}\big\}$ for all $d_S,\, d_I>0$.

Similarly, suppose that $S_{\min}=S(y_0)$ for some $y_0\in\bar\Omega$. Thanks to Lemma \ref{l3.1}(ii), we then obtain
\begin{equation}\label{lamss}
(r_{\min}-S^q(y_0))\beta(y_0)I(y_0)\le S(y_0)-\Lambda_{\min}.
\end{equation}
Hence, if $S(y_0)\ge\Lambda_{\min}$,
$$
\min\big\{\Lambda_{\min},\ r^{\frac{1}{q}}_{\min}\big\}\le S_{\min}.
$$
However, if $S(y_0)<\Lambda_{\min}$, it follows from \eqref{lamss} that $$
S(y_0)>r_{\min}^{\frac{1}{q}}\ge \min\big\{\Lambda_{\min},\ r^{\frac{1}{q}}_{\min}\big\}.
$$
Therefore, we obtain $\min\big\{\Lambda_{\min},\ r_{\min}^{\frac{1}{q}}\big\}\le S_{\min}$ for all $d_S,\, d_I>0$.
\end{proof}

\subsection{Small $d_I$:\ proof of Theorem \ref{TH2.8}{\rm (i)}}
We first investigate the asymptotic profiles of the EE of \eqref{model-2} with $p=1$ as $d_I\to0$ and prove Theorem \ref{TH2.8}(i).

\begin{proof}[Proof of Theorem \ref{TH2.8}{\rm (i)}] Suppose that \eqref{model-2} has an EE $(S,I)$  for $0<d_I\ll 1$.
Define
\begin{equation}\label{D2-aa}
\kappa=d_SS+d_II.
\end{equation}
It is easy to see from \eqref{1-2} that $(\kappa,I)$ solves
\begin{equation}\label{D1}
    \begin{cases}
        d_S\Delta \kappa -\kappa +d_S\Lambda -(d_S\eta-d_I)I=0,\ \ \ & x\in\Omega,\cr
\displaystyle        d_I\Delta I +\Big[\frac{\beta}{d_S^q}(\kappa-d_II)^q-(\eta+\gamma)\Big]I=0, & x\in\Omega,\cr
        \partial_{\nu}\kappa=\partial_{\nu}I=0, &x\in\partial\Omega.
     \end{cases}
\end{equation}
If $d_I<d_S\eta$, it easily follows from the comparison principle for elliptic equations and the first equation of \eqref{D1} that
\begin{equation}\label{D2}
    \kappa< d_S\tilde{S}\ \ \ \mbox{on}\ \bar\Omega,
\end{equation}
where $\tilde{S}$ is the unique solution of \eqref{DFE_mass_bd}. This and \eqref{D2-aa} imply that for all $d_S>0$ and $d_I< d_S\eta_{\min}$,
\begin{equation}\label{D3}
    S< \tilde{S}\ \ \ \mbox{on}\ \bar\Omega.
\end{equation}
So if $d_I<d_S\eta_{\min}$,
\begin{equation}\label{D4}
\begin{cases}
0< d_I\Delta I+(\beta \tilde{S}^q-(\eta+\gamma))I,\ \ \ & x\in\Omega,\cr
0=\partial_{\nu}I, &x\in\partial\Omega.
\end{cases}
\end{equation}
 Integrating over $\Omega$, we get from \eqref{D4} that
$$
\int_\Omega(\beta \tilde{S}^q-(\eta+\gamma))I>0,
$$
which yields $\tilde{\Omega}=\{x\in\bar\Omega:\ \tilde{S}(x)>h^{\frac{1}{q}}(x)\}\ne\emptyset$.

On the other hand, an analysis similar to \cite[Lemma 2.3]{allen2008asymptotic} shows that
$\mathcal{R}_{0}\to\max_{\bar\Omega}\frac{\beta\tilde S^q}{\gamma+\eta}$ as $d_I\to0$.
Thus, by Proposition \ref{theorem_ex}, we see that \eqref{1-2} with $p=1$ possesses at least one EE for all small $d_I$ provided that $\tilde{\Omega}\ne\emptyset$. Therefore, the EE exists for $0<d_I\ll 1$ if and only if $\tilde{\Omega}\ne\emptyset$.

In the sequel, we suppose that $\tilde{\Omega}\ne\emptyset$ and let $(S, I)$ be the EE  for $0<d_I\ll 1$.
Adding the first two equations in \eqref{1-2} and  integrating  over $\Omega$, we obtain
\begin{equation}\label{D6-ab}
\int_\Omega(S+\eta I)=\int_\Omega\Lambda,
\end{equation}
which gives
\begin{equation}\label{D6-ac}
\int_{\Omega}I\le \frac{1}{\eta_{\min}}\int_{\Omega}\Lambda.
\end{equation}
Thanks to \eqref{D4} and \eqref{D6-ac}, we can employ similar arguments as in the proof of \cite[Theorem 2.1]{PSW} to deduce that
\begin{equation}\label{D6}
    \lim_{d_I\to 0}\|I\|_{L^{\infty}(K)}=0\quad \text{for any compact subset} \ K\subset\big\{x: \ \tilde{S}(x)<h^{\frac{1}{q}}(x)\big\}.
\end{equation}
This proves \eqref{D6-2}.

By \eqref{D2}-\eqref{D3}, restricting to a subsequence if necessary, we may assume $\kappa\overset{\ast}{\rightharpoonup} \kappa^*\ \ \text{in } L^\infty(\Omega)$, and
$S\overset{\ast}{\rightharpoonup} S^*$ in  $L^\infty(\Omega)$
as $d_I\to 0$ for some $\kappa^*, S^*\in L^{\infty}(\Omega)$.
Multiplying the first equation of \eqref{1-2}  by $S$ and integrating over $\Omega$, we obtain $d_S\int_\Omega |\nabla S|^2+\int_\Omega S^2\le \int_\Omega \Lambda S+\int_\Omega \gamma SI$. By \eqref{D3} and \eqref{D6-ac}, there exists $C>0$ such that $\|S\|_{W^{1, 2}(\Omega)}\le C$ for all $0<d_I<1$. So restricting to a further subsequence if necessary, we may assume $S^*\in W^{1, 2}(\Omega)$ and $S\to S^*$ weakly in $W^{1, 2}(\Omega)$ as $d_I\to 0$.


We  claim that
\begin{equation}\label{D7}
    \kappa^*=d_SS^*\ \quad \text{a.e. in}\ \Omega.
\end{equation}
To see this, let $\varphi\in L^{\infty}(\Omega)$. By \eqref{D6-ac},
$$
d_{I}\int_{\Omega}|\varphi I|\le d_{I}\|\varphi\|_{L^{\infty}(\Omega)}\int_{\Omega}I\le \frac{d_{I}}{\eta_{\min}}\|\varphi\|_{L^{\infty}(\Omega)}\int_{\Omega}\Lambda\to 0, \quad \text{as}\ d_I\to 0.
$$
Taking $d_I\to 0$ in the equation
$$
\int_{\Omega}\kappa \varphi=d_S\int_{\Omega}S\varphi+d_{I}\int_{\Omega}I\varphi
$$
yields $\int_{\Omega}\kappa^*\varphi=d_S\int_{\Omega}S^*\varphi$. Letting $\varphi=\kappa^*-d_SS^*$, we obtain
$$\int_{\Omega}(\kappa^*-d_SS^*)^2=\int_{\Omega}(\kappa^*-d_SS^*)\varphi=\int_{\Omega}\kappa^*\varphi-d_S\int_{\Omega}S^*\varphi=0, $$
which proves \eqref{D7}.

By \eqref{D6-ac}, the Riesz representation theorem and the Banach-Alaoglu theorem, after passing to a subsequence if necessary, there is a finite  { Radon} measure $\mu$  such that $I\overset{\ast}{\rightharpoonup}\mu\ \ \text{in } [C(\bar\Omega)]^*$ as $d_I\to 0$. We then claim that
\begin{equation}\label{D8}
    \mu(\bar\Omega)=\int_{\Omega}d\mu>0.
\end{equation}
We proceed by contradiction to establish \eqref{D8}. Suppose for contradiction that $\mu(\bar\Omega)=0$.  By \eqref{D6-ab},
$$
\int_{\Omega}S+\int_{\Omega}\eta I=\int_{\Omega}\Lambda,\ \quad \forall\; d_I>0.
$$
Letting $d_I\to 0$ gives
\begin{equation*}
\int_{\Omega}S^*=\int_{\Omega}\Lambda,
\end{equation*}
which along with the fact $\int_{\Omega}\tilde{S}=\int_{\Omega}\Lambda$ yields
\begin{equation}\label{D9}
\int_{\Omega}S^*=\int_{\Omega}\tilde{S}.
\end{equation}
On the other hand, we have from \eqref{D3} that $S^*\le \tilde{S}$ a.e. in $\Omega$. It then follows from \eqref{D9} that $S^*=\tilde{S}$ a.e. in $\Omega$. Consequently, using \eqref{D3}, we can conclude that
\begin{equation}\label{D10}
    \lim_{d_I\to 0}\|\tilde{S}-S\|_{L^1(\Omega)}
    =\lim_{d_I\to 0}\Big(\int_{\Omega}\tilde{S}-\int_{\Omega}S\Big)=0.
\end{equation}
In view of the uniform boundedness of $S$ due to \eqref{D3}, the limit in \eqref{D10} holds in $L^\ell(\Omega)$ for any finite $\ell\ge 1$.  By the second equation of \eqref{1-2} and the variational characterization of the principal eigenvalue of elliptic equations, we have
$$
0\le d_{I}\int_{\Omega}|\nabla \varphi|^2-\int_{\Omega}[\beta S^q-(\eta+\gamma)]\varphi^2, \quad \forall\, \varphi\in W^{1,2}(\Omega), \ d_I>0.
$$
Taking $d_I\to 0$, we obtain that
\begin{equation}\label{D11}
\int_{\Omega}[\beta ({S}^{*})^{q}-(\eta+\gamma)]\varphi^2
\le 0,\ \quad \forall\, \varphi\in W^{1,2}(\Omega),
\end{equation}
which  implies that $S^*\leq h^{\frac{1}{q}}$ a.e. in $\Omega$, and so $\tilde S\leq h^{\frac{1}{q}}$ a.e. in $\Omega$. This contradicts the assumption that $\tilde{\Om}\ne\emptyset$. Therefore, \eqref{D8} holds.

Multiplying the first equation of \eqref{D1} by $\kappa$ and integrating on $\Omega$, we obtain
$$
d_S\int_{\Omega}|\nabla\kappa|^2+\int_{\Omega}\kappa^2=d_S\int_{\Omega}\Lambda \kappa-\int_{\Omega}(d_S\eta-d_I)\kappa I\le d_S^2\int_{\Omega}\Lambda \tilde{S},\quad \forall\; d_I\le d_S\eta_{\min},
$$
where we have used \eqref{D2} in the last step.
So passing to a further subsequence if necessary, we have that $\kappa\to\kappa^* $  weakly in $W^{1,2}(\Omega)$ as $d_I\to0$.

Let $\phi\in W^{1,2}(\Omega)\cap C(\bar\Omega)$. Multiplying the first equation of \eqref{D1} by $\phi$ and integrating over $\Omega$, we obtain
$$
-d_S\int_{\Omega}\nabla \kappa\cdot \nabla \varphi -\int_{\Omega}\kappa\varphi+d_S\int_{\Omega}\Lambda\varphi-d_S\int_{\Omega}\eta\varphi I+d_{I}\int_{\Omega}\varphi I=0.
$$
Letting $d_I\to 0$ and recalling \eqref{D7}, we obtain
\begin{equation*}
    -d_S\int_{\Omega}\nabla S^*\cdot \nabla \varphi -\int_{\Omega}S^*\varphi+\int_{\Omega}\Lambda \varphi -\int_{\Omega}\eta\varphi d\mu=0.
\end{equation*}
Here we used the fact of $I\overset{\ast}{\rightharpoonup}\mu\ \ \text{in } [C(\bar\Omega)]^*$ as $d_I\to 0$. Hence, \eqref{D6-3} holds.

Since $W^{1,2}(\Om)$ is compactly embedded in $L^2(\Om)$ and $S$ is uniformly bounded,  passing to a further subsequence if necessary, we have that $S\to S^*$ in $ L^{\ell}(\Om)$ for any $\ell\ge 1$. Therefore, \eqref{D11} holds and $S^*\le h^{\frac{1}{q}}$ a.e in $\Om$. It is clear from Lemma \ref{lemma_Sbound} that there is $m>0$ such that $S^*>m$ a.e in $\Om$. So \eqref{D6-3-3} holds.

Finally, by Proposition \ref{lemma_fg} in the appendix with $f=S$ and $g=I$, there exist measurable set $F_1$ with Lebesgue measure zero and closed sets $F_l$, $l\ge 2$, with $\Omega=\cup_{l\ge 1} F_l$ such that $S^*$ is continuous on $F_l$ for all $l\ge 2$ and $S^qI \overset{\ast}{\rightharpoonup} (S^*)^q\mu$ in $[C(\Omega')]^*$ with $\Omega':=\cup_{l\ge 2}F_l$. Integrating the equation of $I$ over $\Omega$ and using the fact that $F_1$ has Lebesgue measure zero lead to
\begin{eqnarray*}
\int_\Omega \beta (S^q-h)I=\int_{\Omega'} \beta (S^q-h)I=0.
\end{eqnarray*}
Taking $d_I\to 0$, we obtain
$$
\int_{\Omega'} \beta ([S^*]^q-h)d\mu=\int_{\{x\in\Omega':\ S^*(x)\neq h^{\frac{1}{q}}(x)\}} \beta ([S^*]^q-h)d\mu=0.
$$
This implies $\mu(\{x\in\Omega':\ S^*(x)\neq h^{\frac{1}{q}}(x)\})=0$. In view of $\mu(\bar\Omega)>0$, we have $\mu(\{x\in\Omega': \ S^*(x)= h^{\frac{1}{q}}(x)\}\cup F_1)>0$.
\end{proof}

\subsection{Small $d_S$ and $d_I$:\ proof of Theorem \ref{TH2.10}{\rm (i)}}
We now study the asymptotic profiles of the EE of \eqref{1-2} with $p=1$ when both $d_S$ and $d_I$ are small and prove Theorem \ref{TH2.10}(i).

\begin{proof}[Proof of Theorem \ref{TH2.10}{\rm (i)}] Notice that $\tilde S\to \Lambda$ uniformly on $\bar\Omega$ as $d_S\to 0$.  Similar to \cite[Lemma 2.3]{allen2008asymptotic}, we have
$$
\mathcal{R}_{0}\to\max_{\bar\Omega}\frac{\beta\Lambda^q}{\gamma+\eta}=\max_ {\bar\Omega}\frac{\Lambda^q}{h} \quad \text{as}\ (d_S, d_I)\to (0, 0).
$$
Therefore, by $\hat\Omega:=\{x\in\bar\Omega: \  \Lambda(x)>h^\frac{1}{q}(x)\}\neq\emptyset$ and Proposition \ref{theorem_ex}, \eqref{model-2}  has an EE $(S, I)$ for sufficiently small $d_S$ and $d_I$. Define
$$
u=\frac{d_SS+d_II}{d_S}=S+\frac{d_I}{d_S}I\ \
\mbox{ and}\ \ v=\frac{d_I}{d_S}I.
$$
Then, thanks to \eqref{1-2}, $(u,v)$ satisfies
\begin{equation}\label{C1}
\begin{cases}
\displaystyle d_S\Delta u +\Lambda -u+\big(1-\frac{d_S}{d_I}\eta\big)v=0,\ \ \ &x\in\Om,\cr
\displaystyle d_S\Delta v+\frac{d_S}{d_I}\beta[(u-v)^q-h]v=0, & x\in\Om,\cr
\partial_{\nu}u=\partial_{\nu}v=0, & x\in\partial\Om,\cr
0<v<u, & x\in\Om.
\end{cases}
\end{equation}
Note that if $\frac{d_S}{d_I}\eta<1$, then system \eqref{C1} is cooperative while it is a predator-prey system if $\frac{d_S}{d_I}\eta>1$.

We first derive a uniform upper bound for $u$. Let $x_0\in\bar{\Om}$ such that $u(x_0)=u_{\max}$. It follows from Lemma \ref{l3.1} and the first equation of \eqref{C1} that
\begin{equation}\label{C2}
0\le \Lambda(x_0)-u(x_0)+\Big[1-\frac{d_S}{d_I}\eta(x_0)\Big]v(x_0),
\end{equation}
which is equivalent to
$$
\frac{u(x_0)-v(x_0)}{\eta(x_0)}+\frac{d_S}{d_I}v(x_0)\le \frac{\Lambda(x_0)}{\eta(x_0)}\le \Big(\frac{\Lambda}{\eta}\Big)_{\max}.
$$
This together with the last inequality in \eqref{C1} gives
$$v
(x_0)\le \frac{d_I}{d_S}\Big(\frac{\Lambda}{\eta}\Big)_{\max}.
$$
As a consequence, we obtain from \eqref{C2} that
\begin{equation}\label{C3}
u(x)\le \Lambda_{\max}+\frac{d_I}{d_S}\Big(\frac{\Lambda}{\eta}\Big)_{\max},\ \ \ \forall \;x\in\bar\Omega.
\end{equation}
Noticing $0<v<u$, there exist  $u^*,v^*\in L^{\infty}(\Om)$ such that, restricting to a subsequence if necessary,
   \begin{equation}\label{C4}
       u\overset{\ast}{\rightharpoonup} u^*\ \ \text{in } L^\infty(\Omega),\quad  v\overset{\ast}{\rightharpoonup} v^*\ \ \text{in } L^\infty(\Omega),\quad \text{and} \quad 0\le v^*\le u^*\ \ \text{a.e.\ in}\ \Om,
   \end{equation}
 as $d_I+|\frac{d_I}{d_S}-\sigma|\to 0$.

Let $\varphi\in W^{1,2}(\Om)$. Multiplying the first equation of \eqref{C1} by $\varphi$ and integrating over $\Omega$, we obtain
$$
d_S\int_{\Om}u\Delta \varphi +\int_{\Om}\Big[\Lambda-u+\Big(1-\frac{d_S}{d_I}\eta\Big)v\Big]\varphi=0.
$$
Taking $d_I+|\frac{d_I}{d_S}-\sigma|\to 0$, we deduce
$$
\int_{\Om}\Big[\Lambda-u^*+\Big(1-\frac{\eta}{\sigma}\Big)v^*\Big]\varphi=0.
$$
Since $\varphi\in W^{1,2}(\Om)$ is arbitrary,
   \begin{equation}\label{C5}
       \Lambda -u^* +\Big(1-\frac{\eta}{\sigma}\Big)v^*=0 \quad \text{a.e. in}\ \Om.
   \end{equation}
Since $S=u-v$ and $I=\frac{d_S}{d_I}v$,   \eqref{TH9-eq-1} holds with $S^*=u^*-v^*$ and $I^*=\frac{1}{\sigma}v^*$. It is clear from Lemma \ref{lemma_Sbound} and  \eqref{TH9-eq-1} that
$$
\min\{\Lambda_{\min},\ r_{\min}^{\frac{1}{q}}\}\le S^*\le \Lambda_{\max} \ \ \ \mbox{a.e.\ in\ $\Om$.}
$$
Here, we have used the assumption $\hat{\Om}\not=\emptyset$ and $\Lambda_{\max}>r_{\max}^{\frac{1}{q}}$. Hence, \eqref{TH9-eq-2} holds.

Let $0<\varepsilon<\min\{\eta_{\min},\sigma\}$ be given. Set
$$\tilde{\Lambda}=\Lambda, \ \ \tilde{\eta}=\min\{\eta,\sigma\}-\varepsilon, \ \ \text{and} \ \  \tilde{\gamma}=\gamma+\eta-\tilde{\eta}.
$$
Then,
$$\sigma>\tilde{\eta}_{\max} \ \ \text{and} \ \ \tilde{h}:=\frac{\tilde{\gamma}+\tilde{\eta}}{\beta}=h.
$$
We observe that $(u,v)$ is a subsolution of \eqref{C9}. {Note also from Lemma \ref{lemma_Sbound} that $(u,v)\in C(\bar{\Om} : \mathcal{R}_*)$, where $ \mathcal{R}_*$ is defined by \eqref{R-rectangle-star}.  Since the orbit of every solution of \eqref{C9-1} with initial data in $C(\bar{\Om}:\mathcal{R}_*)$ is precompact, then there is a solution $(\tilde{u},\tilde{v})\in C(\bar{\Om}:\mathcal{R}_*)$ of \eqref{C9} such that $(u,v)\le (\tilde{u},\tilde{v})$.}  By Proposition \ref{prop-6.2} and the comparison principle for cooperative systems, we obtain
   $$
   \limsup_{d_I+|\frac{d_I}{d_S}-\sigma|\to0}(u,v)\le \Big(\min\{\Lambda,h^{\frac{1}{q}}\}+\frac{\sigma}{\tilde{\eta}}(\Lambda-h^{\frac{1}{q}})_+,\
   \frac{\sigma}{\tilde\eta}(\Lambda-h^{\frac{1}{q}})_+\Big) \quad \text{uniformly on}\ \bar\Om.
   $$
   Letting $\varepsilon\to 0$ and recalling that $S=u-v$ and $I=\frac{d_S}{d_I}v$, we deduce that \begin{equation*}
0\le    I^*\le \frac{1}{\min\{\sigma,\eta\}}(\Lambda-h^{\frac{1}{q}})_+\ \ \quad \text{a.e.\ in}\ \Om.
   \end{equation*}

To complete the proof of \eqref{TH9-eq-3}, it remains to show $\int_{\Om}I^*>0$, which is equivalent to $\int_{\Om}v^*>0$. Suppose by contradiction that $\int_{\Om}v^*=0$. Then $u^*=\Lambda$ a.e in $\Om$ by \eqref{TH9-eq-1}.  Multiplying \eqref{DFE_mass_bd} and the first equation of \eqref{C1} by $\tilde{S}-u$,  taking the difference and integrating over $\Omega$, we obtain
\begin{eqnarray*}
   \int_{\Om}(\tilde{S}-u)^2&=&-d_S\int_{\Om}|\nabla(\tilde{S}-u)|^2
   -\int_{\Om}\Big(1-\frac{d_S}{d_I}\eta\Big)v(\tilde{S}-u)\\
   &\le& \Big(1+\eta_{\max}\frac{d_S}{d_I}\Big)\|\tilde{S}-u\|_{L^\infty(\Omega)}\int_{\Om}v.
\end{eqnarray*}
Using $\int_{\Om}v\to\int_{\Om}v^*=0$ and \eqref{C3}, we obtain that $\|\tilde{S}-u\|_{L^2(\Om)}\to 0$ as $d_I+|\frac{d_I}{d_S}-\sigma|\to0$.  Since $\|\tilde{S}-\Lambda\|_{L^\infty(\Omega)}\to 0$ as $d_S\to0$,
we have that
$$
\mbox{$\|u-\Lambda\|_{L^2(\Om)}\to 0$\,\ \ as\ $d_I+|\frac{d_I}{d_S}-\sigma|\to0.$}
$$
As a result, we obtain from \eqref{C3} that $\|u-\Lambda\|_{L^\ell(\Om)}\to 0$ as  $d_I+|\frac{d_I}{d_S}-\sigma|\to0$ for any $\ell\in[1,\infty)$. By the uniform boundedness of $v$,  $\|v\|_{L^{\ell}(\Om)}\to 0$  as  $d_I+|\frac{d_I}{d_S}-\sigma|\to0$ for any $l\in[1,\infty)$. Consequently, we can conclude that
\begin{equation}\label{C1-0}
\mbox{$u-v\to \Lambda$\ \ in $L^{\ell}(\Om)$\,\ \ as\ $d_I+|\frac{d_I}{d_S}-\sigma|\to0$\ \ for any $\ell\in[1,\infty)$.}
\end{equation}
Consider the following eigenvalue problem:
\begin{equation}\label{C1-ei}
\begin{cases}
\displaystyle d_S\Delta \varphi+\frac{d_S}{d_I}\beta[(u-v)^q-h]\varphi=\lambda\varphi, & x\in\Om,\cr
\partial_{\nu}\varphi=0, & x\in\partial\Om.
\end{cases}
\end{equation}
Note that $v$ is a positive eigenfunction of \eqref{C1-ei} with principal eigenvalue being zero.
By the variational characterization of the principal eigenvalue, we know that
\begin{equation}\label{C6}
       \int_{\Om}\beta((u-v)^q-h)\psi^2\le d_I\int_{\Om}|\nabla \psi|^2, \quad \forall\, \psi\;\in W^{1,2}(\Om).
   \end{equation}
 Taking $d_I+|\frac{d_I}{d_S}-\sigma|\to 0$,  by \eqref{C1-0}, we have  $\int_{\Om}\beta(\Lambda^q-h)\psi^2\le 0$ for all $\psi\in W^{1,2}(\Om)$. Whence, we have $\Lambda\le h^{\frac{1}{q}}$ a.e. in $\Om$, which contradicts the assumption that $\tilde{\Om}\ne\emptyset$. Therefore, $\int_{\Om}I^*>0$.

If $q=1$,    we have $\int_{\Om}\beta(u^*-v^*-h)\varphi^2\le 0$ by taking $d_I+|\frac{d_I}{d_S}-\sigma|\to 0$ in \eqref{C6}. Hence,
   \begin{equation}\label{C7}
       u^*-v^*-h\leq0\  \quad \text{a.e.\ in}\ \Om.
   \end{equation}
Adding  \eqref{C5} and \eqref{C7} yields
   \begin{equation}\label{C8}
       I^*=\frac{1}{\sigma}v^*\ge \frac{1}{\eta}(\Lambda-h)_{+} \  \quad \text{a.e.\ in}\ \Om.
   \end{equation}
By \eqref{C7}, $S^*=u^*-v^*\le h$ a.e. in $\Om$, and so \eqref{TH9-eq-4} holds.

Finally,  suppose that $\eta\le \sigma$. Let $0<\varepsilon\ll 1$ be fixed. We rewrite \eqref{C1} as
\begin{equation}\label{C10-3}
\begin{cases}
\displaystyle d_S\Delta u +\Lambda-\varepsilon\frac{d_S}{d_I}v -u+\Big[1-\frac{d_S}{d_I}(\eta-\varepsilon)\Big]v=0,\ \ \ &x\in\Om,\cr
\displaystyle d_S\Delta v+\frac{d_S}{d_I}\beta\Big[(u-v)^q-\frac{((\gamma+\varepsilon)+(\eta-\varepsilon))}{\beta}\Big]v=0, & x\in\Om,\cr
\partial_{\nu}u=\partial_{\nu}v=0, & x\in\partial\Om,\cr
0<v<u, & x\in\Om.
\end{cases}
\end{equation}
In Proposition \ref{prop-6.2}, taking $\tilde{\Lambda}=\Lambda$, $\tilde{\gamma}=\gamma+\varepsilon$, $\tilde{\eta}=\eta-\varepsilon$  and $\tilde{h}=\frac{\tilde{\eta}+\tilde{\gamma}}{\beta}$,
then $\{x\in \Om : \tilde\Lambda(x)>\tilde{h}^{\frac{1}{q}}(x)\}$ is not empty and $\sigma>\tilde{\eta}_{\max}$. It is easy to see that $(u,v)$ is a subsolution of \eqref{C9}.
On the other hand, choose sufficiently large constant $C_0$ such that
$\frac{\Lambda\sigma}{\tilde\eta}<C_0$, $u<C_0+(\frac{1}{2}\tilde h_{\min})^{\frac{1}{q}}$ and $v<C_0$. Then,
$$
\Big(C_0+\Big(\frac{1}{2}\tilde h_{\min}\Big)^{\frac{1}{q}},\ C_0\Big)
$$
is a supersolution of \eqref{C9}. Since \eqref{C9} is a cooperative system, by the standard iteration argument of sub-super solutions, it  follows from Proposition \ref{prop-6.2} that
$$
\limsup_{d_I+|\frac{d_I}{d_S}-\sigma|\to 0}(u,v)\le \Big(\min\{\Lambda,\tilde{h}^{\frac{1}{q}}\}+\frac{\sigma}{\tilde{\eta}}(\Lambda-\tilde{h}^{\frac{1}{q}})_+,\ \frac{\sigma}{\tilde{\eta}}(\Lambda-\tilde{h}^{\frac{1}{q}})_+\Big)
$$
uniformly on $\bar\Omega$. Taking $\varepsilon\to 0$,  it holds that
       \begin{equation}\label{C10-4}
\limsup_{d_I+|\frac{d_I}{d_S}-\sigma|\to 0}(u,v)\le \Big(\min\{\Lambda,{h}^{\frac{1}{q}}\}+\frac{\sigma}{{\eta}}(\Lambda-{h}^{\frac{1}{q}})_+,\ \frac{\sigma}{{\eta}}(\Lambda-{h}^{\frac{1}{q}})_+\Big)
    \end{equation}
uniformly on $\bar\Omega$.

Next,  let
$$
\tilde{\Lambda}=\Lambda-\varepsilon(\sigma+1)
\Big[\Lambda_{\max}+(\sigma+1)\Big(\frac{\Lambda}{\eta}\Big)_{\max}\Big], \ \ \tilde{\eta}=\eta-\varepsilon\ \ \text{and}\ \ \tilde{\gamma}=\gamma+\varepsilon,
$$
where $0<\varepsilon\ll 1$.
Then $\{x\in\Om:\ \tilde{\Lambda}(x)>\tilde{h}^{\frac{1}{q}}(x)\}$ is not empty. By \eqref{C10-3}, $(u,v)$ is a supersolution of \eqref{C9}.
On the other hand, for any small $\epsilon_0>0$ satisfying $\epsilon_0<\min\{u,\ \tilde\Lambda\}$, $(\epsilon_0,0)$ is a subsolution of \eqref{C9}. As before, by the sub-super solution argument, we can deduce from Proposition \ref{prop-6.2} that
$$
\liminf_{d_I+|\frac{d_I}{d_S}-\sigma|\to 0}(u,v)\ge \Big(\min\{\tilde{\Lambda},\tilde{h}^{\frac{1}{q}}\}+\frac{\sigma}{\tilde{\eta}}(\Lambda-\tilde{h}^{\frac{1}{q}})_+,\ \frac{\sigma}{\tilde{\eta}}(\Lambda-\tilde{h}^{\frac{1}{q}})_+\Big) \quad \text{uniformly on}\ \bar\Om.
$$
Taking $\varepsilon\to 0$, we obtain
\begin{equation}\label{C10-5}
\liminf_{d_I+|\frac{d_I}{d_S}-\sigma|\to 0}(u,v)\ge \Big(\min\{\Lambda,{h}^{\frac{1}{q}}\}+\frac{\sigma}{{\eta}}(\Lambda-{h}^{\frac{1}{q}})_+,\ \frac{\sigma}{{\eta}}(\Lambda-{h}^{\frac{1}{q}})_+\Big)
\end{equation}
uniformly on $\bar\Omega$. Combining \eqref{C10-4}-\eqref{C10-5}, we have
$$
(u, v)\to\Big(\min\{\Lambda,{h}^{\frac{1}{q}}\}+\frac{\sigma}{{\eta}}(\Lambda-{h}^{\frac{1}{q}})_+,\ \frac{\sigma}{{\eta}}(\Lambda-{h}^{\frac{1}{q}})_+\Big)
$$
uniformly on $\bar\Omega$ as $d_S+|\frac{d_I}{d_S}-\sigma|\to0$. This proves Theorem \ref{TH2.10}(i).
\end{proof}

\vskip25pt
\section{Proofs of the main results: case of $0<p<1$} In this section, we assume that $0<p<1$ and prove Theorem \ref{TH2.8}(ii), Theorem \ref{TH2.9}(ii) and Theorem \ref{TH2.10}(ii).  By Proposition \ref{theorem_ex},  \eqref{model-2} has an EE $(S, I)$ for all $d_S, d_I>0$.

\subsection{Three useful lemmas} We first prepare some useful results.

\begin{lem}\label{Lem-5-1} Suppose that $0<p<1$ and $q>0$. Let $(S,I)$ be an EE of \eqref{model-2}. Then it holds that
\begin{equation}\nonumber
\Big[\Big(\frac{\beta}{\gamma+\eta}\Big)_{\min}S_{\min}^q\Big]^{\frac{1}{1-p}}\leq I_{\min}\leq I_{\max}\leq \Big[\Big(\frac{\beta}{\gamma+\eta}\Big)_{\max}S_{\max}^q\Big]^{\frac{1}{1-p}}.
\end{equation}
\end{lem}
\begin{proof} Suppose that $I(x_0)=I_{\max}$ for some $x_0\in\bar\Omega$. Then
by Lemma \ref{l3.1}(i), $\De I(x_0)\leq0$. So it follows from the second equation of \eqref{1-2} that
$$
-\beta(x_0)S^q(x_0)I^p(x_0)+(\gamma(x_0)+\eta(x_0))I(x_0)\le 0.
$$
Thus, we have
$$
I_{\max}\leq \Big[\Big(\frac{\beta}{\gamma+\eta}\Big)_{\max}S_{\max}^q\Big]^{\frac{1}{1-p}}.
$$

Similarly,  suppose that $I(y_0)=I_{\min}$ for some $y_0\in\bar\Omega$. Then, $\De I(y_0)\geq0$ by Lemma \ref{l3.1}(ii). So
\begin{equation}\nonumber
-\beta(y_0)S^q(y_0)I^p(y_0)+(\gamma(y_0)+\eta(y_0))I(y_0)\ge 0,
\end{equation}
which implies
$$
\Big[\Big(\frac{\beta}{\gamma+\eta}\Big)_{\min}S_{\min}^q\Big]^{\frac{1}{1-p}}\leq I_{\min}.
$$
\end{proof}

\begin{lem}\label{Lem-5-2} Suppose that $0<p<1$ and $q>0$.  Let $(S,I)$ be any EE of \eqref{1-2}. Then it holds that
\begin{equation}\label{1-1b}
S_{\max}\leq \left(1+\frac{d_I}{d_S\eta_{\min}}\right)\Lambda_{\max}
\end{equation}
and
\begin{equation}\label{1-1c}
I_{\max}\leq \left(\frac{d_S}{d_I}+\frac{1}{\eta_{\min}}\right)\Lambda_{\max}.
\end{equation}
\end{lem}
\begin{proof} Let $w=d_SS+d_II$. Then $w$ solves
\begin{equation}\label{w_bd}
\begin{cases}
\De w+\Lambda-S-\eta I=0,\ \ \ & x\in\Omega,\cr
\partial_{\nu}w=0, & x\in\p\Omega.
\end{cases}
\end{equation}

Suppose that $w(x_0)=w_{\max}$ for $x_0\in\bar\Omega$. Then $\De w(x_0)\leq0$. So by \eqref{w_bd}, we have
$$\Lambda(x_0)-S(x_0)-\eta(x_0) I(x_0)\geq0.$$
This shows
$$
S(x_0)\leq\Lambda_{\max}\ \ \text{and}\ \ I(x_0)\leq\frac{\Lambda_{\max}}{\eta_{\min}}.
$$
As a result, we have
$$
d_SS_{\max}\leq w(x_0)=d_SS(x_0)+d_II(x_0)\leq d_S\Lambda_{\max}+d_I\frac{\Lambda_{\max}}{\eta_{\min}}
$$
and
$$
d_II_{\max}\leq w(x_0)=d_SS(x_0)+d_II(x_0)\leq d_S\Lambda_{\max}+d_I\frac{\Lambda_{\max}}{\eta_{\min}},
$$
which  imply the desired results.
\end{proof}

\begin{lem}\label{Lem-5-1d} Suppose that $0<p<1$ and $q>0$.   Let $c_0$ denote the unique positive solution of the algebraic equation
$$
 \frac{\Lambda_{\min}}{1+\left(\frac{d_S}{d_I}+\frac{1}{\eta_{\min}}\right)^p\Lambda^p_{\max}\beta_{\max}}=c+c^q.
$$
Then any EE $(S,I)$ of \eqref{model-2} satisfies
\begin{equation}\label{eq:lem-5-1d}
S_{\min}\ge c_0\quad \text{and}\quad
I_{\min}\ge\left[\Big(\frac{\beta}{\eta+\gamma}\Big)_{\min}c_0^q\right]^\frac{1}{1-p}.
\end{equation}

\end{lem}
\begin{proof} Suppose that $S(y_0)=S_{\min}$ for $y_0\in\bar\Omega$. As in Lemma \ref{Lem-5-1}, $\Delta S(y_0)\ge 0$ and
$$
\Lambda(y_0)+\gamma(y_0)I(y_0)\le S(y_0)+\beta(y_0)S^q(y_0)I^p(y_0).
$$
It then follows that
$$
\Lambda_{\min}\le (1+\beta_{\max}I_{\max}^p)[S(y_0)+S^q(y_0)].
$$
Hence, in view of Lemma \ref{Lem-5-2}, we have that
$$
\frac{\Lambda_{\min}}{1+\left(\frac{d_S}{d_I}+\frac{1}{\eta_{\min}}\right)^p\Lambda^p_{\max}\beta_{\max}}\le S_{\min}+S_{\min}^q.
$$
Since the function $[0,\infty)\ni c\mapsto c+c^q$ is strictly increasing, $S_{\min}\ge c_0$. The second inequality in \eqref{eq:lem-5-1d} follows from Lemma \ref{Lem-5-1}.
\end{proof}

\subsection{Small $d_I$:\ proof of Theorem \ref{TH2.8}{\rm (ii)}}
We are now in a position to prove Theorem \ref{TH2.8}{\rm (ii)}.

\begin{proof}[Proof of Theorem \ref{TH2.8}{\rm (ii)}] By means of \eqref{1-1b} in Lemma \ref{Lem-5-2},
there exists $C_0>0$ such that $S(x)\leq C_0$ for all $x\in\bar{\Omega}$ and
 $0<d_I\leq1$. Then by Lemma \ref{Lem-5-1}, there exists $C_1>0$ such that
$I(x)\leq C_1$ for all $x\in\bar{\Omega}$ and $0<d_I\leq1$.
By  the $S$-equation in  \eqref{1-2}  and the standard elliptic estimates (\cite{GT}), restricting to a subsequence if necessary, $S\to S_*$ in $C^1(\bar\Omega)$ as $d_I\to0$ for some $S_*\geq0$.

We rewrite the second equation in \eqref{1-2} as
\begin{equation}\nonumber
    d_I\De I+[\beta S^q-(\gamma+\eta)I^{1-p}]I^p=0, \ \ x\in\Omega;\ \ \
        \partial_{\nu}I=0,\ \ x\in\p\Omega.
\end{equation}
Since $S\to S_*$ in $C^1(\bar\Omega)$,  a standard singular perturbation argument shows that $I\to I_*$ uniformly on $\bar\Omega$ as $d_I\to0$, where $I_*$ is given by \eqref{S-lim-1}. It is not hard to see that $S_*$ is the unique nonnegative classical solution of \eqref{S-lim-2}. The maximum principle for elliptic equations implies that $S_*>0$ on $\bar\Omega$, and so $I_*>0$ on $\bar\Omega$ by \eqref{S-lim-1}. The uniqueness of $S_*$ implies that the convergence is independent of the chosen subsequence.
\end{proof}

\subsection{Small $d_S$:\ proof of Theorem \ref{TH2.9}{\rm (ii)}}
In this subsection, we  prove Theorem \ref{TH2.9}{\rm (ii)}.

\begin{proof}[Proof of Theorem \ref{TH2.9}{\rm (ii)}] By\eqref{1-1c}  of Lemma \ref{Lem-5-2}, there exists $C_0>0$ such that $I\leq C_0$ on $\bar\Omega$ for all $0<d_S\leq1$. Suppose that $S(x_0)=S_{\max}$ for some $x_0\in\bar\Omega$. Then $\De S(x_0)\leq0$, and the $S$-equation in  \eqref{1-2} gives
$$
S(x)\leq S(x_0)\leq\Lambda(x_0)+\gamma(x_0) I(x_0)\leq \Lambda_{\max}+\gamma_{\max}C_0:=C_1,\ \  \ \forall x\in\bar\Omega
$$
for all $0<d_S\leq1$. By a standard compactness argument applied to the $I$-equation in  \eqref{1-2}, restricting to a sequence if necessary,
$I\to I_*$ in $C^1(\bar\Omega)$ as $d_S\to0$ for some $I_*\geq0$.

We now claim that $I_*>0$ on $\bar\Omega$. Suppose to the contrary that $I_*(y_0)=0$ for some $y_0\in\bar\Omega$. Since $I\to I_*$ in $C^1(\bar\Omega)$,  $I_{\min}\to0$ as $d_S\to0$. This and Lemma \ref{Lem-5-1} yield $S_{\min}\to0$ as $d_S\to0$. Suppose that $S(z_0)=S_{\min}$ for some $z_0\in\bar\Omega$. We infer from the $S$-equation in  \eqref{1-2} that
$$
0\leftarrow S(z_0)+\beta S^q(z_0)I^p(z_0)\geq\Lambda(z_0)+\gamma(z_0) I(z_0)\geq \Lambda_{\min}>0\ \ \mbox{as}\ d_S\to0,
$$
which is a contradiction. Hence, $I_*>0$ on $\bar\Omega$.

Since $I_*>0$ on $\bar\Omega$,  a standard singular perturbation argument  applied to the $S$-equation in  \eqref{1-2} allows us to conclude that $S\to S_*$ uniformly on
$\bar\Omega$, where $S_*$ satisfies
$$
\Lambda- S_*-\beta S_*^q I_*^p+\gamma I_*=0,\ \ x\in\bar\Omega.
$$
It is easy to see that $I_*$ solves \eqref{I-lim}. The proof is  complete.
\end{proof}

\subsection{Small $d_S$ and $d_I$:\ proof of Theorem \ref{TH2.10}{\rm (ii)}}
Finally, we consider the case where both $d_S$ and $d_I$ are small. As a first step, we need the following result.

 \begin{lem}\label{l3.2} Suppose that $0<p<1$ and $q>0$.  Let $\sigma>\eta_{\max}$ be fixed.
 \begin{itemize}
     \item[\rm (i)]  Let $\underline{u}_0=\Lambda$  and $\underline{v}_0=0$. Define a sequence of nonnegative functions $\{(\underline{u}_n,\underline{v}_n)\}$ inductively as follows:\ if $(\underline{u}_n,\underline{v}_n)$ is defined, let  $\underline{v}_{n+1}$ denote the unique  positive solution of
     $$\underline{u}_n=\underline{v}_{n+1}+h^{\frac{1}{q}}\left(\frac{\underline{v}_{n+1}}{\sigma}\right)^{\frac{1-p}{q}}
     $$
     and $\underline{u}_{n+1}=\Lambda+(1-\frac{\eta}{\sigma})\underline{v}_{n+1}$.  Then $\{(\underline{u}_n,\underline{v}_{n})\}_{n\ge 1}$ is increasing and converges to $(u^*,v^*)$ uniformly on $\bar{\Om}$, where $v^*$ is the unique positive solution of
     \begin{equation}\label{v-limit-equ}
         \Lambda=\frac{\eta}{\sigma}v^*+h^{\frac{1}{q}}\left(\frac{v^*}{\sigma}\right)^{\frac{1-p}{q}}
     \end{equation}
     and $u^*=\Lambda+(1-\frac{\eta}{\sigma})v^*$.

     \item[\rm (ii)] Let $\overline{u}_0=\overline{v}_0=\Lambda_{\max}+(1+\sigma)(\frac{\Lambda}{\eta})_{\max}$. Define a sequence of positive functions $\{(\overline{u}_n,\overline{v}_n)\}$ inductively:\ if $(\overline{u}_n,\overline{v}_n)$ is defined, let  $\overline{u}_{n+1}=\Lambda+(1-\frac{\eta}{\sigma})\overline{v}_n$  and $\overline{v}_{n+1}$ be the unique  positive solution of
     $$
     \overline{u}_{n+1}=\overline{v}_{n+1}+h^{\frac{1}{q}}\left(\frac{\overline{v}_{n+1}}{\sigma}\right)^{\frac{1-p}{q}}.
     $$
     Then $\{(\overline{u}_n,\overline{v}_{n})\}_{n\ge 1}$ is decreasing and converges to $(u^*,v^*)$ uniformly on $\bar{\Om}$, where $(u^*,v^*)$ is the same as in ${\rm(i)}$.
 \end{itemize}

 \end{lem}
 \begin{proof}{\rm (i)} We proceed by induction  to show the monotonicity of $\{(\underline{u}_n,\underline{v}_{n})\}$. Since $\underline{u}_0=\Lambda>0$, then $\underline{v}_1>0=\underline{v}_0$. This in turn yields $\underline{u}_1=\Lambda+(1-\frac{\eta}{\sigma})\underline{v}_1>\Lambda=\underline{u}_0$. Suppose that the monotonicity holds up to $n\ge 1$. Then,
 $$
 \underline{v}_n+h^{\frac{1}{q}}\left(\frac{\underline{v}_n}{\sigma}\right)^{\frac{1-p}{q}}
 =\underline{u}_{n-1}<\underline{u}_n=\underline{v}_{n+1}+h^{\frac{1}{q}}
 \left(\frac{\underline{v}_{n+1}}{\sigma}\right)^{\frac{1-p}{q}}.
 $$
 Since the mapping $[0,\infty)\ni c\mapsto c+h^{\frac{1}{q}}\left(\frac{c}{\sigma}\right)^{\frac{1-p}{q}}$ is strictly increasing, we have $\underline{v}_{n}<\underline{v}_{n+1}$. This in turn implies that $$
 \underline{u}_{n+1}=\Lambda+\big(1-\frac{\eta}{\sigma}\big)\underline{v}_{n+1}
 >\Lambda+\big(1-\frac{\eta}{\sigma}\big)\underline{v}_n=\underline{u}_{n}.
 $$
 Therefore by induction, $\{(\underline{u}_{n},\underline{v}_n)\}$ is increasing.

 Next, we show that this sequence is bounded. Indeed, observe from its monotonicity that
 \begin{eqnarray*}
 \underline{u}_{n+1}=\Lambda +\big(1-\frac{\eta}{\sigma}\big)\underline{v}_{n+1}&=&\Lambda +\big(1-\frac{\eta}{\sigma}\big)\Big[\underline{u}_n-h^{\frac{1}{q}}
 \left(\frac{\underline{v}_{n+1}}{\sigma}\right)^{\frac{1-p}{q}}\Big]\\
 &<&\Lambda +(1-\frac{\eta}{\sigma})\underline{u}_{n+1},\quad \forall\; n\ge 0.
 \end{eqnarray*}
Thus, $\underline{u}_{n+1}< \sigma\Big(\frac{\Lambda}{\eta}\Big)_{\max}$ for every $n\ge 0$. This implies
$$
\underline{v}_{n+1}=\underline{u}_n-h^{\frac{1}{q}}
\left(\frac{\underline{v}_{n+1}}{\sigma}\right)^{\frac{1-p}{q}}<\underline{u}_n
<\sigma\Big(\frac{\Lambda}{\eta}\Big)_{\max},\quad \forall\; n\ge 0,
$$
which shows that $\{(\underline{u}_{n},\underline{v}_n)\}$ is  bounded.

Since $\{(\underline{u}_{n},\underline{v}_n)\}$ is monotone and bounded, it converges pointwise to a pair of bounded functions $(\underline{u}_{\infty},\underline{v}_{\infty})$ on $\bar{\Omega}$, which satisfy
 $$
 \underline{u}_{\infty}=\underline{v}_\infty+h^{\frac{1}{q}}\left(\frac{\underline{v}_{\infty}}{\sigma}\right)^{\frac{1-p}{q}}\quad \text{and}\quad  \underline{u}_{\infty}=\Lambda+\big(1-\frac{\eta}{\sigma}\big)\underline{v}_{\infty}.
 $$
 This shows that $\underline{v}_{\infty}$ is a positive solution of the algebraic equation \eqref{v-limit-equ}. Therefore, we must have $\underline{v}_{\infty}=v^*$ and $\underline{u}_{\infty}={u}^*$. Since $({u}^*,v^*)$ is continuous on $\bar{\Omega}$, we conclude from the Dini's theorem that $(\underline{u}_n,\underline{v}_n)\to({u}^*,v^*)$ uniformly on $\bar\Omega$ as $n\to\infty$.

 \quad {\rm (ii)} Similar to the arguments in {\rm (i)}, it suffices to show that $\{(\overline{u}_n,\overline{v}_n)\}$ is decreasing. We proceed by induction. Observe that
 $$
 \overline{u}_1-\overline{u}_0=\Lambda-\frac{\eta}{\sigma}\overline{u}_0
 \leq\Big[\frac{\Lambda}{\eta}-\frac{1+\sigma}{\sigma}\Big(\frac{\Lambda}{\eta}\Big)_{\max}\Big]\eta
 <0
 $$
 and
 $$
 \overline{v}_1=\overline{u}_1-h^{\frac{1}{q}}
 \left(\frac{\overline{v}_1}{\sigma}\right)^{\frac{1-p}{q}}<\overline{u}_0=\overline{v}_0.
 $$
 Suppose that the monotonicity of $\{(\overline{u}_n,\overline{v}_n)\}$ holds up to some $n\ge 1$. Then we have
 $$
 \overline{u}_{n+1}=\Lambda+\big(1-\frac{\eta}{\sigma}\big)\overline{v}_n
 <\Lambda+\big(1-\frac{\eta}{\sigma}\big)\overline{v}_{n-1}=\overline{u}_n.
 $$
 Since the mapping $[0,\infty)\ni c\mapsto c+\left(\frac{c}{\sigma}\right)^{\frac{1-p}{q}}$ is strictly increasing and
 $$
 \left[\overline{v}_{n}+h^{\frac{1}{q}}\left(\frac{\overline{v}_n}{\sigma}\right)^{\frac{1-p}{q}}\right]
 -\left[\overline{v}_{n+1}+h^{\frac{1}{q}}\left(\frac{\overline{v}_{n+1}}{\sigma}\right)^{\frac{1-p}{q}}\right]=\overline{u}_{n}-\overline{u}_{n+1}>0,
 $$
 we  obtain  $\overline{v}_n>\overline{v}_{n+1}$. Therefore,  $\{(\overline{u}_n,\overline{v}_n)\}_{n\ge 0}$ is decreasing.
 \end{proof}

Now we are ready to establish Theorem \ref{TH2.10}(ii).

\begin{proof}[Proof of Theorem \ref{TH2.10}{\rm (ii)}] By Lemmas \ref{Lem-5-2}-\ref{Lem-5-1d}, restricting to a subsequence if necessary, we have $S\to S^*$ and $I\to I^*$ weakly-star in $L^{\infty}(\Omega)$ for some  $S^*,\,I^*\in L^{\infty}(\Omega)$ satisfying  \eqref{model-di-ds-eq-2}-\eqref{model-di-ds-eq-3} as $d_{I}+|\frac{d_{I}}{d_{S}}-\sigma|\to0$.
Let $\kappa=d_{S}S+d_{I}I$. Then $\kappa$ solves
\begin{equation}\nonumber
    \De \kappa+\Lambda-S-\eta I=0, \ \ x\in\Omega;\ \ \
        \partial_{\nu}\kappa=0,  \ \ x\in\p\Omega.
\end{equation}
Clearly, $\kappa\to \Lambda$ uniformly on $\bar\Omega$ as $d_{I}+|\frac{d_{I}}{d_{S}}-\sigma|\to0$. By a test function argument, it is easily seen that \eqref{model-di-ds-eq-1} holds. Clearly, \eqref{model-di-ds-eq-2}-\eqref{model-di-ds-eq-3} follow from \eqref{model-di-ds-eq-1} and Lemma \ref{Lem-5-1d}.

Finally, suppose that $\sigma>\eta_{\max}$. Let $u=S+\frac{d_I}{d_S}I$ and $v=\frac{d_I}{d_S}I$. Then $(u,v)$ satisfies
\begin{equation}\label{C1-p2}
        \begin{cases}
\displaystyle            d_S\Delta u +\Lambda -u+\left(1-\frac{d_S}{d_I}\eta\right)v=0,\ \ \ &x\in\Om,\cr
\displaystyle            d_S\Delta v+\left(\frac{d_S}{d_I}\right)^{p}\beta\Big[(u-v)^q-h\left(\frac{d_S}{d_I}v\right)^{1-p}\Big]v^p=0, & x\in\Om,\cr
            \partial_{\nu}u=\partial_{\nu}v=0, & x\in\partial\Om,\cr
            0<v<u, & x\in\Om.
        \end{cases}
    \end{equation}
Since  $\sigma>\eta_{\max}$, $\frac{d_S}{d_I}\eta<1$ and \eqref{C1-p2} is cooperative if $d_{I}+|\frac{d_I}{d_S}-\sigma|$ is small. Moreover,  $d_S\Delta u +\Lambda -u<0$ and ${u}\ge{\tilde{S}}$  by the comparison principle if $d_{I}+|\frac{d_I}{d_S}-\sigma|$ is small.

 Since ${\tilde{S}}\to{\Lambda}=\underline{u}_0$ uniformly on $\bar\Om$ as $d_S\to0$,
it  follows from the first equation of \eqref{C1-p2} that
$$
\liminf_{d_{I}+|\frac{d_I}{d_S}-\sigma|\to 0}u\ge \Lambda=\underline{u}_0\ \ \ \  \text{uniformly on}\ \bar\Om.
$$
Clearly, we also have
$$
\liminf_{d_{I}+|\frac{d_I}{d_S}-\sigma|\to 0}v\ge \underline{v}_0=0\ \ \ \ \text{uniformly on}\ \bar\Om.
$$
We can now proceed as in the proof of Lemma \ref{Appendix-lem6}  to show that
    \begin{equation}  \label{ZZ1-1}
    \liminf_{d_{I}+|\frac{d_I}{d_S}-\sigma|\to 0}u\ge\underline{u}_n \quad \text{and}\quad \liminf_{d_{I}+|\frac{d_I}{d_S}-\sigma|\to 0}v\ge\underline{v}_n \ \ \  \ \text{uniformly on}\ \bar\Om,
    \end{equation}
where $(\underline{u}_n,\underline{v}_n)$ is given by Lemma \ref{l3.2}(i).

Note that $(u,v)$ satisfies \eqref{C3} and $0<v<u$. So there is $d_0>0$ such that for all  $d_I+|\frac{d_I}{d_S}-\sigma|<d_0$,
  $$
  u<\Lambda_{\infty}+(\sigma+1)\Big(\frac{\Lambda}{\eta}\Big)_{\max}=\overline{u}_0 \quad \text{and}\quad v<\overline{u}_0=\overline{v}_0,\quad \forall\, x\in\bar\Omega.
  $$
Thus, we can also follow the procedure as in the proof of Lemma \ref{Appendix-lem6} to show that
\begin{equation} \label{ZZ1-2}
\limsup_{d_{I}+|\frac{d_I}{d_S}-\sigma|\to 0}u\le\overline{u}_n \quad \text{and}\quad \limsup_{d_{I}+|\frac{d_I}{d_S}-\sigma|\to 0}v\le\overline{v}_n\ \ \ \ \text{uniformly on}\ \bar\Om,
\end{equation}
where $(\overline{u}_n,\overline{v}_n)$ is given by Lemma \ref{l3.2}(ii). In view of \eqref{ZZ1-1}, \eqref{ZZ1-2} and Lemma \ref{l3.2}, we obtain the desired result.
\end{proof}

\section{Numerical simulations and conclusions}\label{Nuremical-Sim-section}

\subsection{Simulations}
In this subsection, we simulate \eqref{model-2} with the objective to illustrate the proved results and  complement the cases that have not rigorously proven yet.

We take the domain $\Omega$ to be a unit circle centered the origin. The initial data will always be  $S_0=0.8$ and $I_0=0.2$.  Choose $\Lambda=1$ such that $\tilde S=1$ is the unique solution of $-d_S\Delta = \Lambda- S$ with homogeneous Neumann boundary condition.  We only consider the case $p=1$ and $q=0.5$.
{ We will solve the time-dependent problem \eqref{model-2} instead of the steady state problem \eqref{1-2} with $d_S$ or $d_I$ being small. Model \eqref{model-2} will be solved on a sufficiently large time interval $[0, T]$ so that the variation of $(S(\cdot, t), I(\cdot, t))$ is very small when $t\approx T$.  Then $(S(\cdot, T), I(\cdot, T))$ will be presented below as an approximation of the asymptotic limit of EE solutions of \eqref{1-2} as $d_S\to 0$ or $d_I\to 0$. We  use the finite element package in Matlab with a mesh size of 1453 nodes for the domain $\Omega$. The programs for the simulations can be found at \url{https://github.com/YixiangMath/Diffusive-SIS-Model-.git}. }

 \begin{figure}
     \centering
     \begin{subfigure}[b]{0.48\textwidth}
         \centering
          \includegraphics[scale=.3]{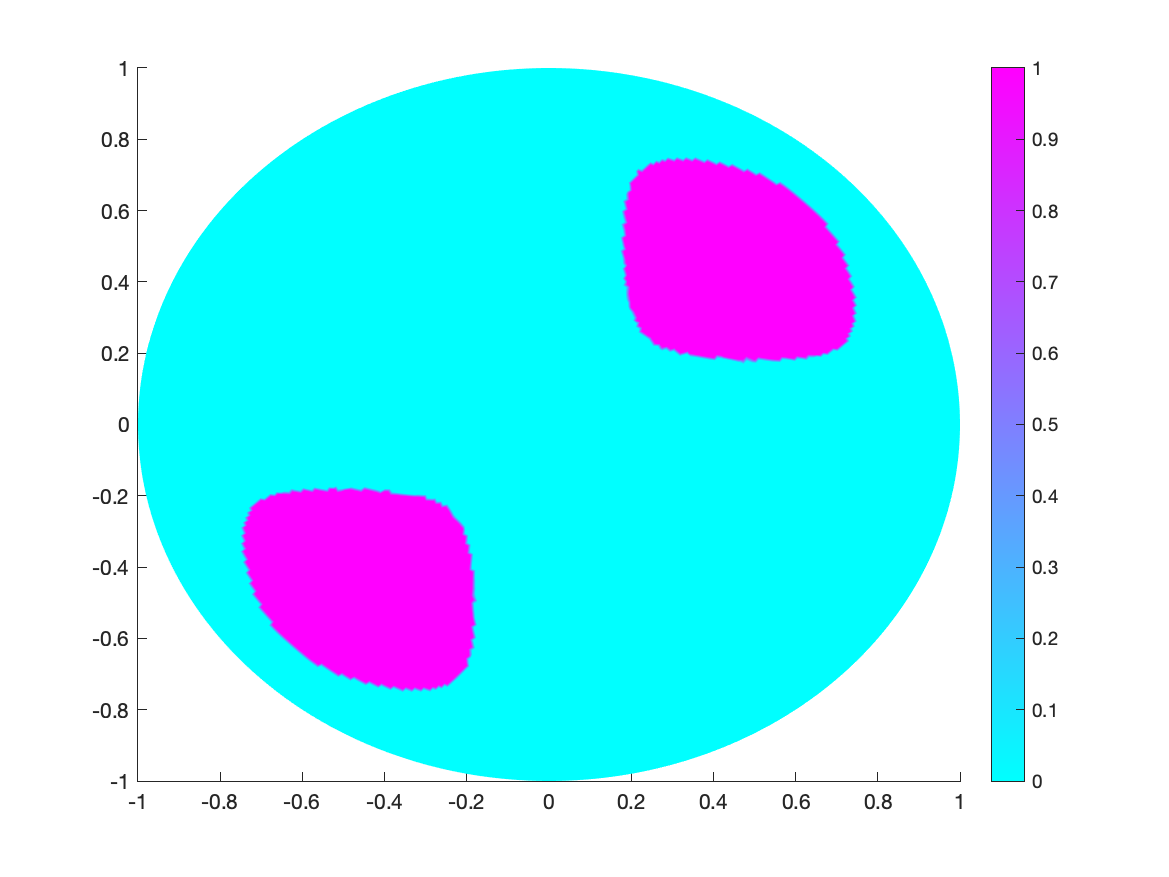}
         \caption{$\chi_{\{(x, y): h^{1/q}(x, y)-S(x, y, 400)<10^{-4}\}}$ when $d_S=1, d_I=10^{-5}$.}
         \label{fig:a8}
     \end{subfigure}
     \begin{subfigure}[b]{0.48\textwidth}
        \centering
     \includegraphics[scale=.3]{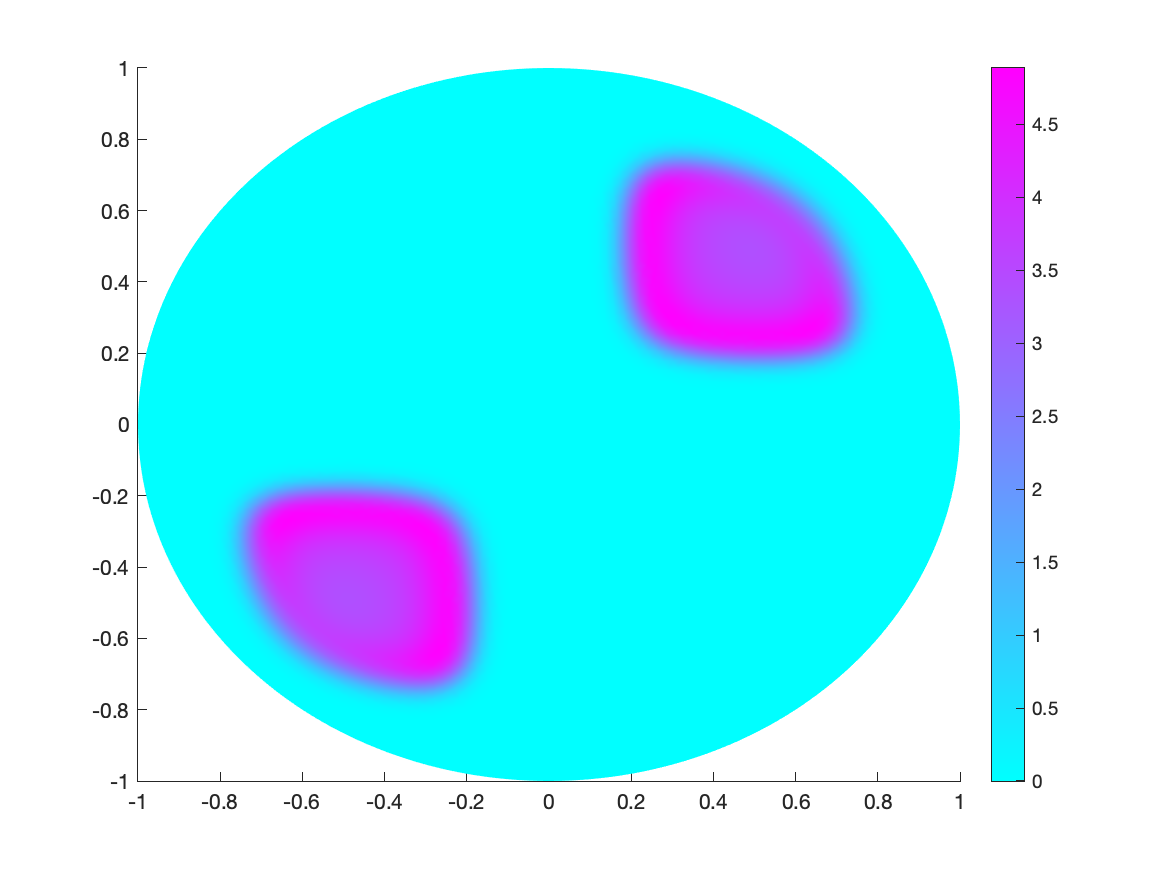}
             \caption{$I(x, y, 400)$ when $d_S=1, d_I=10^{-5}$.}
         \label{fig:b8}
     \end{subfigure}

\vskip25pt
         \begin{subfigure}[b]{0.48\textwidth}
         \centering
          \includegraphics[scale=.3]{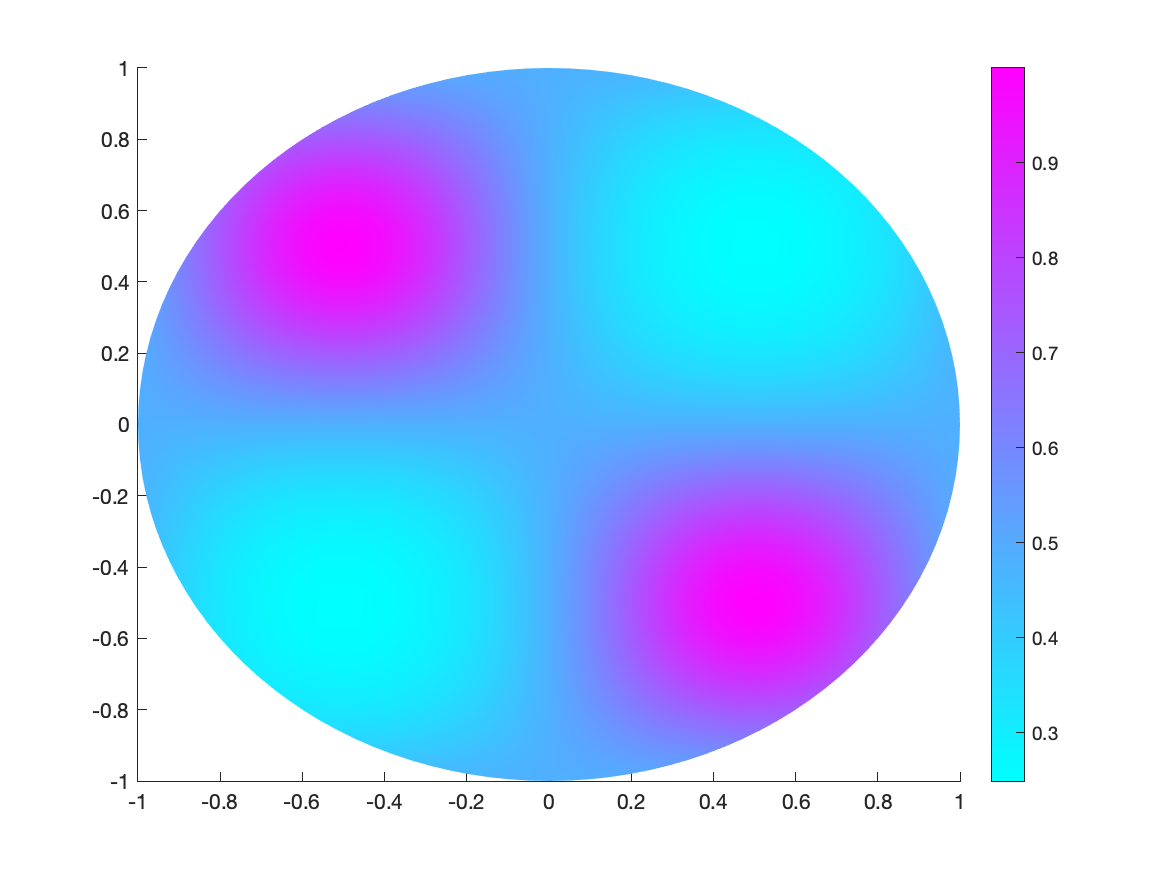}
         \caption{$S(x, y, 400)$ when $d_S=10^{-5}, d_I=1$.}
         \label{fig:a9}
     \end{subfigure}
     \begin{subfigure}[b]{0.48\textwidth}
        \centering
     \includegraphics[scale=.3]{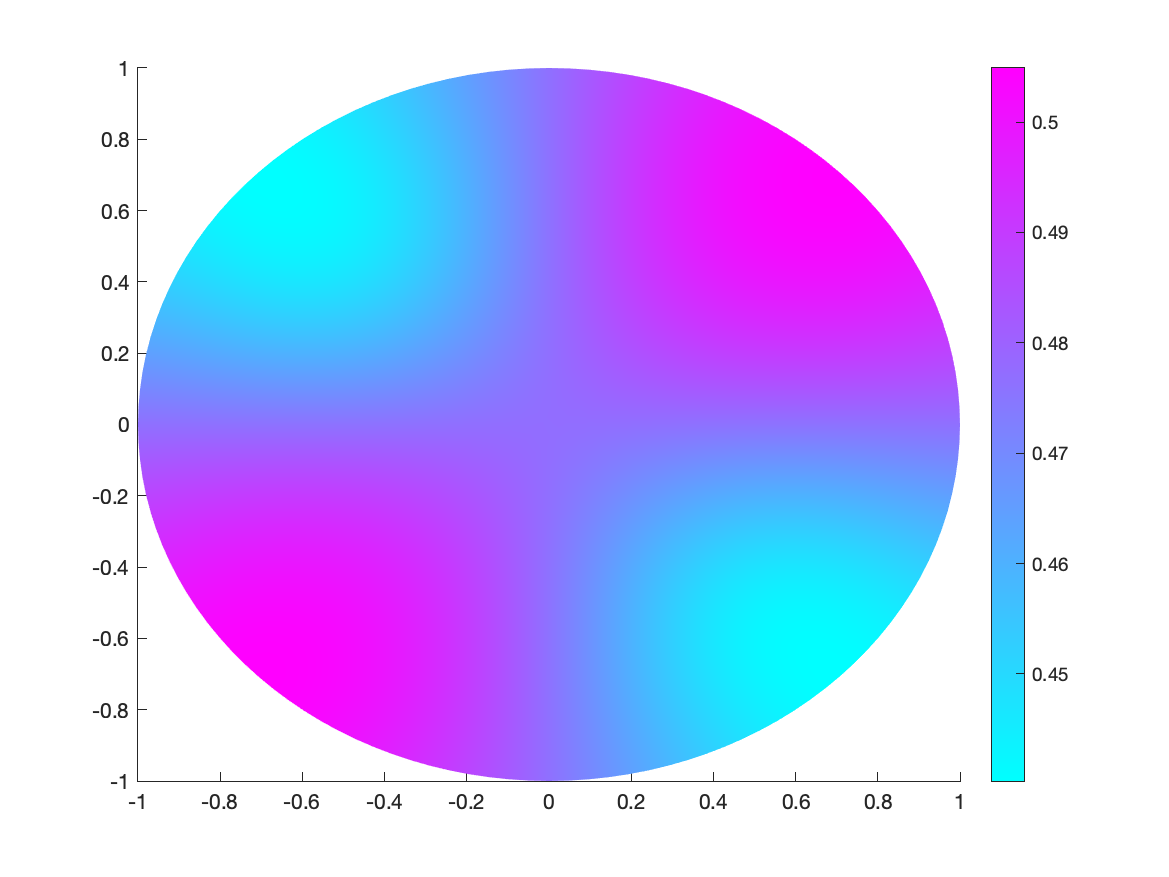}
             \caption{$I(x, y, 400)$ when $d_S=10^{-5}, d_I=1$.}
         \label{fig:b9}
     \end{subfigure}

\vskip25pt
         \begin{subfigure}[b]{0.48\textwidth}
         \centering
          \includegraphics[scale=.3]{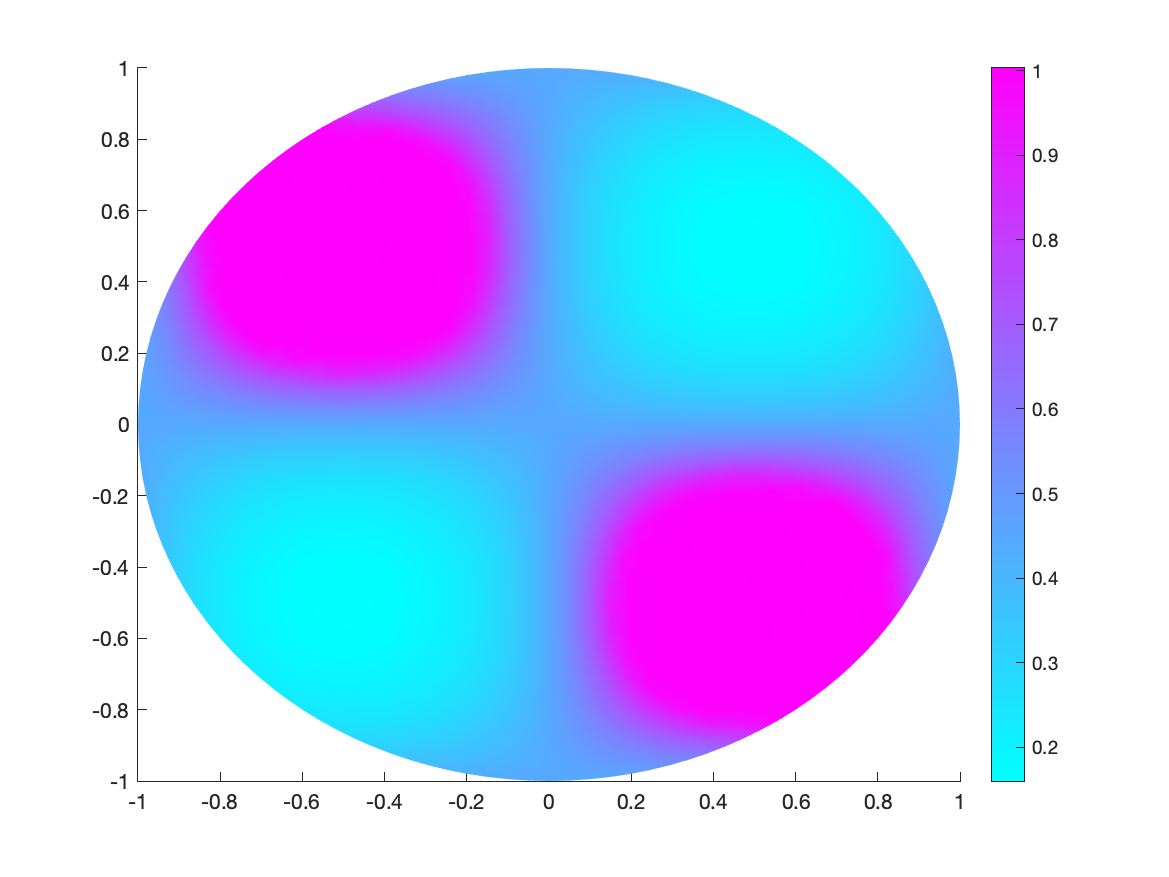}
         \caption{$S(x, y, 400)$ when $d_S=d_I=10^{-5}$.}
         \label{fig:a10}
     \end{subfigure}
     \begin{subfigure}[b]{0.48\textwidth}
        \centering
     \includegraphics[scale=.3]{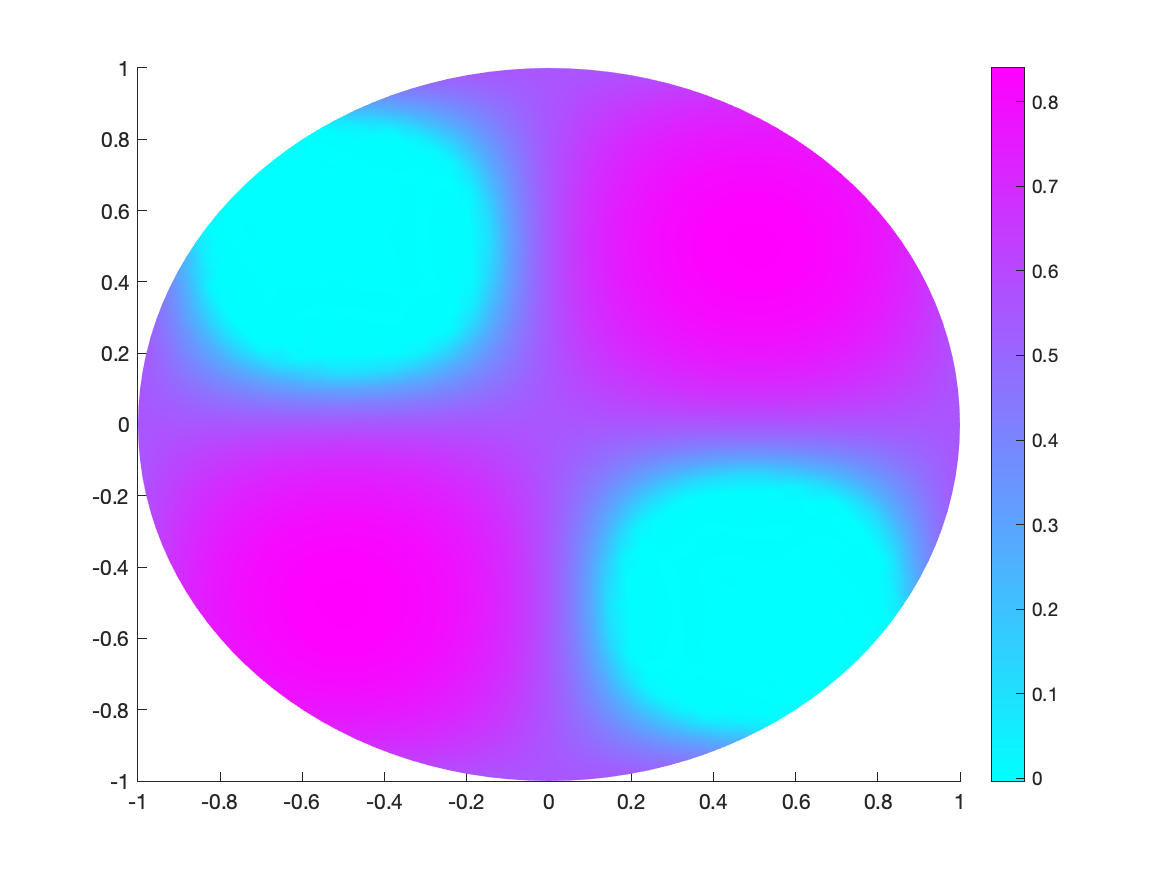}
             \caption{$I(x, y, 400)$ when $d_S=d_I=10^{-5}$.}
         \label{fig:b10}
     \end{subfigure}
    \caption{Simulation of model \eqref{model-2} with $p=1$ and $q=0.5$. The disease transmission, recovery, and mortality rates are $
\beta(x, y)=1.5+\sin(\pi x)\sin(\pi y)$, $\gamma(x, y)=1$, and $\eta(x, y)=1$,  respectively.}
     \label{fig:3}
\end{figure}


\textit{First simulation.}  The disease transmission, recovery and disease induced mortality rates are
$$
\beta(x, y)=3+2\sin(\pi x)\sin(\pi y) \ \ \text{and} \ \ \gamma(x, y)=\eta(x, y)=1,
$$
respectively.  The minimum of $h(x, y)=(\gamma(x, y)+\eta(x, y))/\beta(x, y)$ on $\bar\Omega$ is attained at $(0.5, 0.5)$ and $(-0.5, -0.5)$, which can be treated as the points of the highest risk.

According to Theorem \ref{TH2.8}(i), if we restrict the movement of infected individuals, the EE $(S, I)$ will have a subsequence that converges weakly to $(S^*, \mu)$, where $S^*\le h^\frac{1}{q}$ and $\mu$ resides in the set $\{(x,y)\in\Omega: S^*(x, y)\neq h^\frac{1}{q}(x, y)\}$. This is shown in Figures \ref{fig:3}(a)-(b) with parameters $d_S=1$ and $d_I=10^{-5}$. The red regions in Figure \ref{fig:3}(a) represent $\{(x,y)\in\Omega: S^*(x, y)\neq h^\frac{1}{q}(x, y)\}$, while Figure \ref{fig:3}(b) shows the distribution of infected individuals in these regions. It is worth noting that the points $(0.5, 0.5)$ and $(-0.5, -0.5)$ are included in these regions. 

On the other hand, if we restrict the movement of susceptible individuals ($d_S=10^{-5}$, $d_I=1$), Figures \ref{fig:3}(c)-(d) demonstrate that infected individuals are uniformly distributed spatially, and the density of susceptible individuals is lower in regions with lower risk compared to those with higher risk. This observation aligns with the findings of Theorem \ref{TH2.9}.

When both susceptible and infected individuals have restricted movement ($d_S= d_I=10^{-5}$), as depicted in Figures \ref{fig:3}(e)-(f), it appears that the infection in  lower risk regions is eradicated, and the density of susceptible individuals is higher in lower risk regions compared to higher risk regions. This agrees with Theorem \ref{TH2.10}(i).

\textit{Second simulation.}  The disease transmission, recovery and disease induced mortality rates are
$$
\beta(x, y)=0.5, \ \gamma(x, y)=f(x)f(y), \ \ \text{and} \ \ \eta(x, y)=0.1,
$$
respectively, where
\begin{equation}\label{f}
    f(x)=
    \left\{
\begin{array}{ll}
  0.5+0.4x^2,   & x\le 0, \\
  0.5,   & 0<x\le 0.25,\\
  0.5+0.4(x-0.25)^2, & 0.25<x\le 0.5,\\
  0.5+1.6(x-0.625)^2,& 0.5<x\le 1.
\end{array}
\right.
\end{equation}
 The minimum of $h(x, y)=(\gamma(x, y)+\eta(x, y))/\beta(x, y)$ on $\bar\Omega$ is attained at
 $$
 \Omega^*=[0, 0.25]\times [0, 0.25]\cup (\{0.625\}\times [0, 0.25])\cup  ([0, 0.25]\times \{0.625\})\cup \{(0.625, 0.625)\},
 $$
 which is the set of points of the highest risk. Here, $\Omega^*$ consists with a point, two line segments, and one rectangular region.
 If we limit the movement  of infected people ($d_S=1, d_I=10^{-5}$), as shown in Figures \ref{fig:5}(a)-(b), the infected people live  near $\Omega^*$. In Figures \ref{fig:5}(c)-(d), we plot $\chi_{\{(x, y): h^{1/q}(x, y)-S(x, y, 400)<10^{-2}\}}$ and $\chi_{\{(x, y): h^{1/q}(x, y)-S(x, y, 400)<10^{-4}\}}$ ($\chi$ is the characteristic function), respectively. So the red regions in Figures \ref{fig:5}(c)-(d)  are approximations of $E:=\{(x,y)\in\Omega: S^*(x, y)\neq h^\frac{1}{q}(x, y)\}$. By Theorem \ref{TH2.8}(i), when the movement of infected individuals is restricted, the infected individuals will remain within $E$. This is  supported by the evidence presented in Figure \ref{fig:5}(b). Additionally, it is worth noting that these regions contain $\Omega^*$, and the infected individuals predominantly reside in the vicinity of $\Omega^*$. However, it is unclear what is the shape of $E$.

 \begin{figure}
     \centering
     \begin{subfigure}[b]{0.48\textwidth}
         \centering
          \includegraphics[scale=.3]{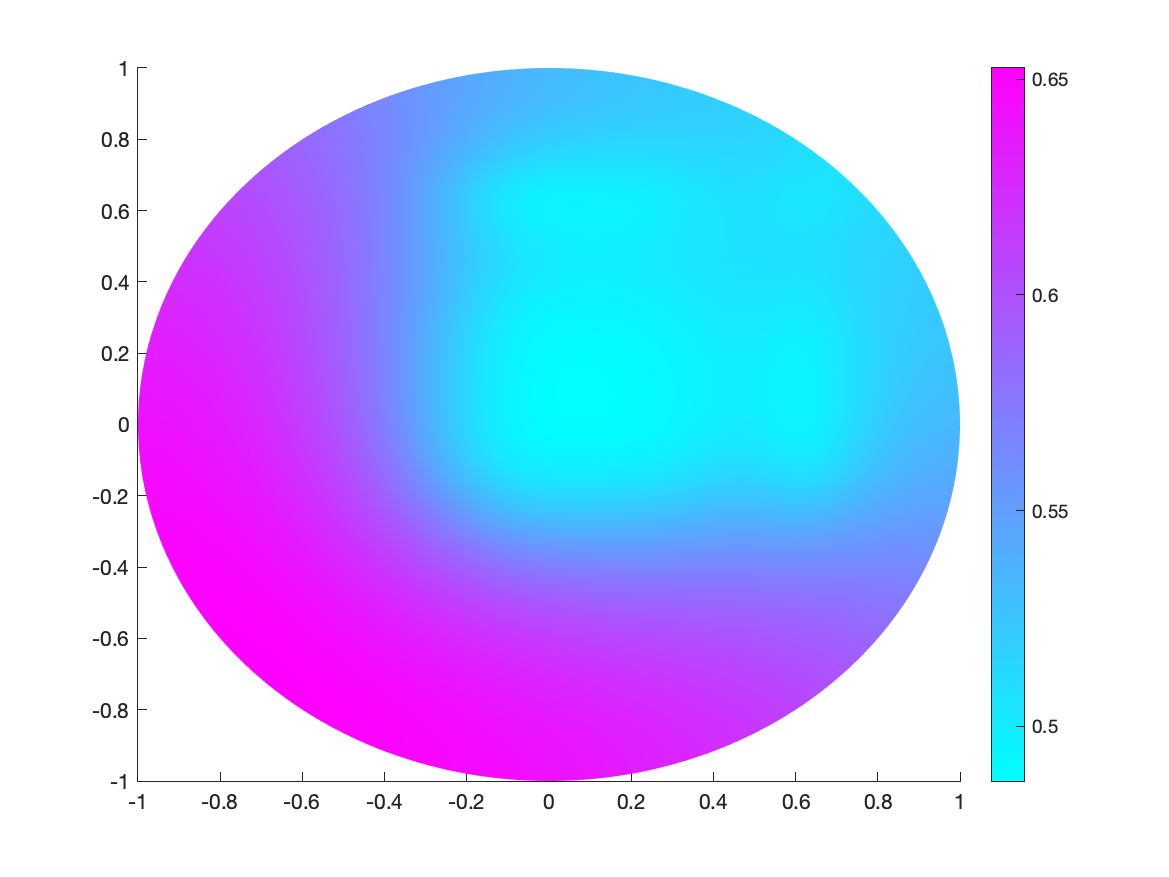}
        \caption{$S(x, y, 800)$}
         \label{fig:a14}
     \end{subfigure}
     \begin{subfigure}[b]{0.48\textwidth}
        \centering
     \includegraphics[scale=.3]{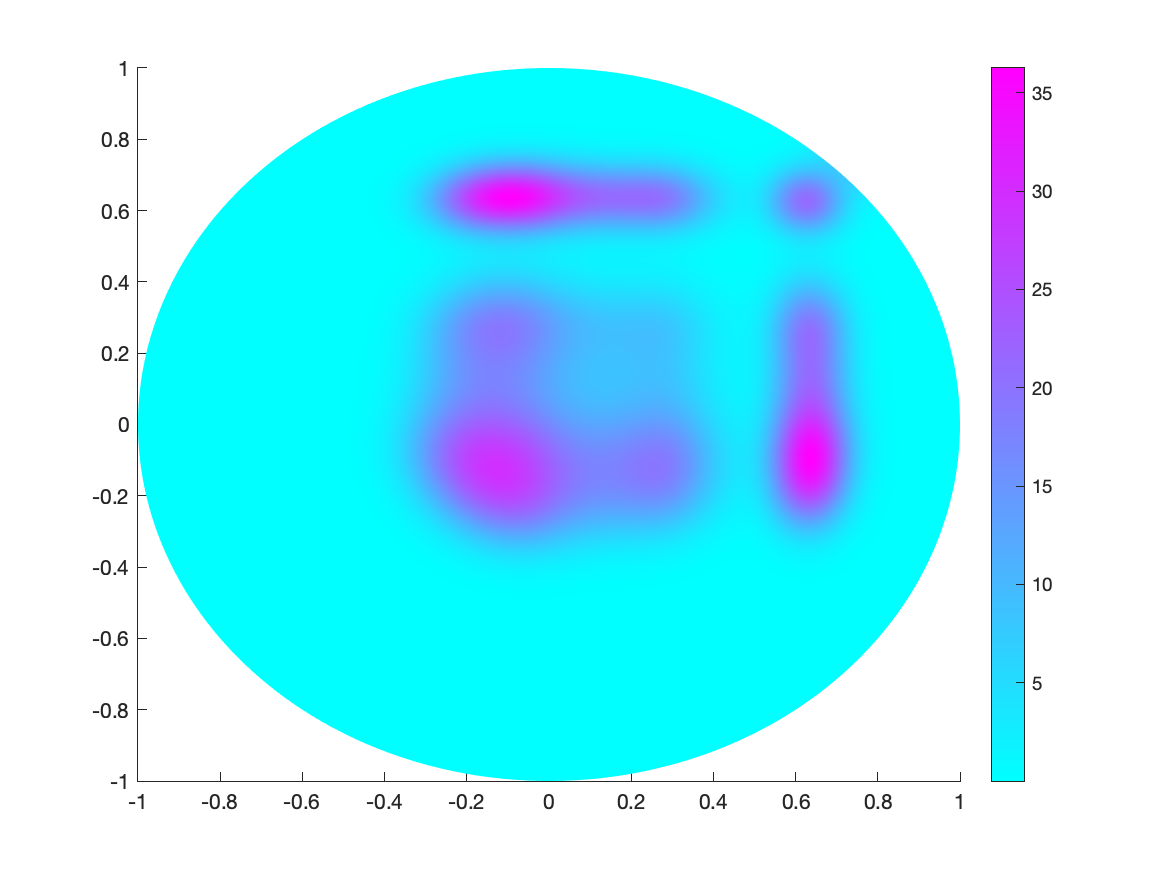}
             \caption{$I(x, y, 800)$}
         \label{fig:b14}
     \end{subfigure}

               \begin{subfigure}[b]{0.48\textwidth}
         \centering
          \includegraphics[scale=.3]{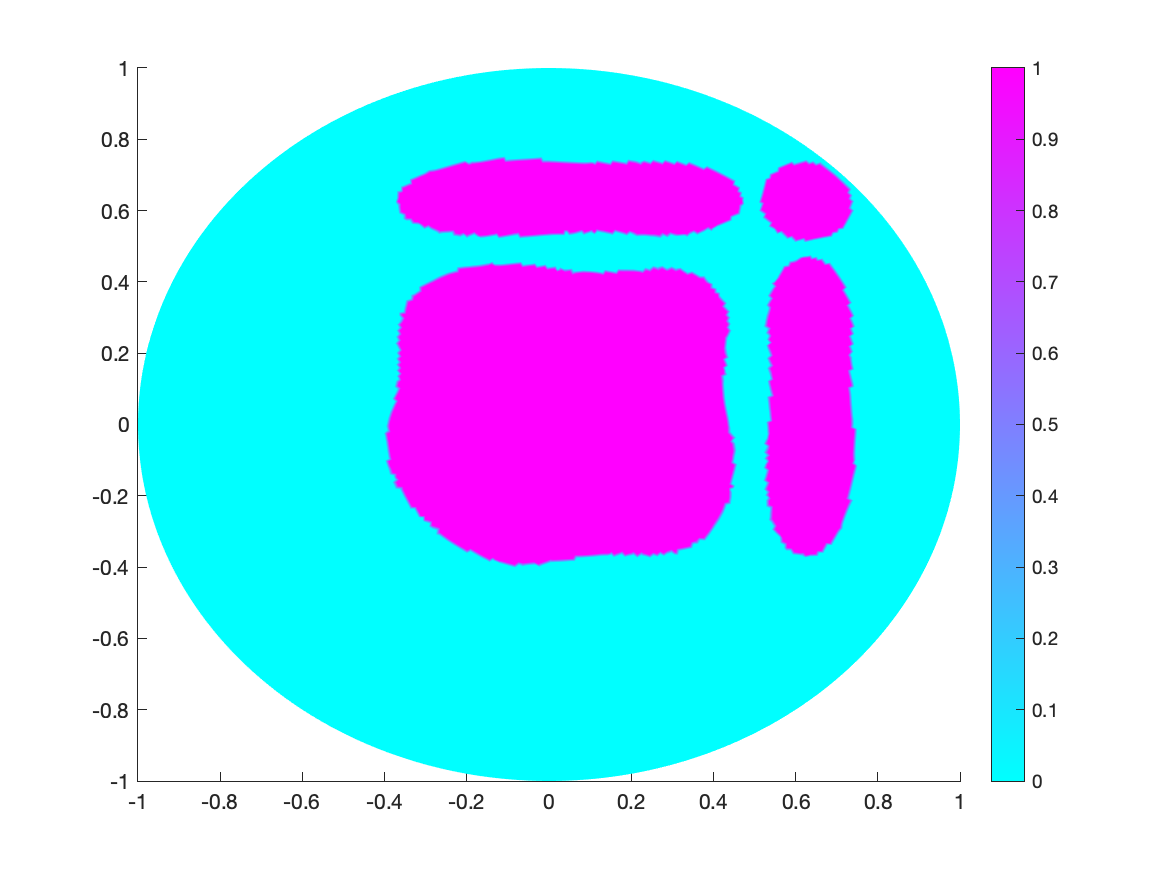}
        \caption{$\chi_{\{(x, y): h^{1/q}(x, y)-S(x, y, 800)<10^{-2}\}}$}
         \label{fig:a15}
     \end{subfigure}
     \begin{subfigure}[b]{0.48\textwidth}
        \centering
     \includegraphics[scale=.3]{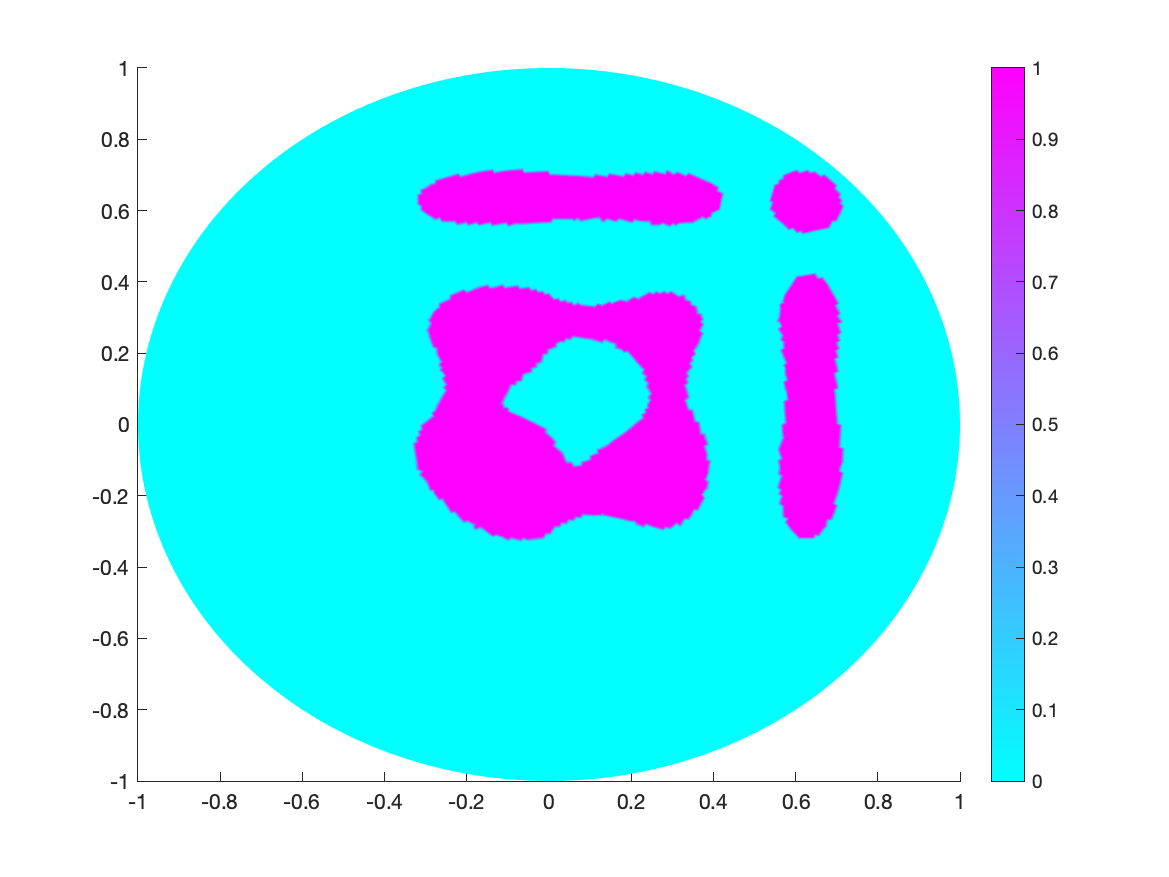}
             \caption{$\chi_{\{(x, y): h^{1/q}(x, y)-S(x, y, 800)<10^{-4}\}}$}
         \label{fig:b15}
     \end{subfigure}

     \caption{Simulation of model \eqref{model-2} with $p=1$, $q=0.5$, $d_S=1$, and $d_I=10^{-5}$. The disease transmission, recovery and mortality rates are $
\beta(x, y)=0.5$, $\gamma(x, y)=f(x)f(y)$, and $\eta(x, y)=0.1$ respectively, where $f$ is given by \eqref{f}. The minimum of $h=(\gamma+\eta)/\beta$ is attained at  $[0, 0.25]\times [0, 0.25]\cup (\{0.625\}\times [0, 0.25])\cup  ([0, 0.25]\times \{0.625\})\cup \{(0.625, 0.625)\}$, which consist with the points of the highest risk.}
     \label{fig:5}
\end{figure}

\subsection{Conclusions}

In this subsection, we discuss and compare the results for models \eqref{model-1} and \eqref{model-2} from the perspective of disease control. We start by considering the case $p=1$, i.e., the infection rate is proportional to the number of infected people.

Firstly, we consider the impact of limiting the movement of infected people (i.e., $d_I\to 0$). If the individual birth and  death rates have a negligible effect on the disease dynamics (i.e., model \eqref{model-1}), \cite[Theorem 2.1, Proposition 2.1]{PSW} imply  that the infected people will reside at the points of the highest risk (i.e. the points at which $\beta(x)/\gamma(x)$ attain the maximum). In the case that there are only finitely many points that are of the highest risk, a Dirac delta type concentration phenomenon appears.
However, when the susceptible individuals have a linear growth term and the disease-induced mortality is significant (i.e., model \eqref{1-2}), Theorem \ref{TH2.8}(i) indicates that the infected individuals will reside in the region $\{x\in\bar\Omega: S^*(x)=h^\frac{1}{q}(x)\}$. In this case, unlike in model \eqref{model-1}, the infected individuals may not concentrate at the highest-risk points when the highest-risk habitat consists of finitely many isolated points (see Figures \ref{fig:3}(a)-(b)).

Next, we consider the impact of limiting the movement of susceptible people  (i.e., $d_S\to 0$).
For model \eqref{model-1}, \cite[Theorem 2.2(i)]{PSW} states that the total population size $N$ has a significant effect on the disease dynamics: if $N$ is small ($N\le \int_{\Omega}r^\frac{1}{q}$), the disease will be eliminated; if $N$ is large ($N>\int_{\Omega}r^\frac{1}{q}$), the disease cannot be controlled. In contrast, for model \eqref{model-2}, Theorem \ref{TH2.9} implies that the disease can never be eliminated by limiting only the movement of susceptible individuals. This is confirmed by Figures \ref{fig:3}(c)-(d).

Then, we consider the impact of limiting the movement of both susceptible and infected people (i.e. $d_I+|\frac{d_I}{d_S}-\sigma|\to 0$). For both models \eqref{model-1} and \eqref{model-2},  \cite[Theorem 2.3(i)]{PSW} and Theorem \ref{TH2.10} (i) indicate that the infection will be eliminated only in some region of the domain. For model \eqref{model-2}, these results are supported by Figures \ref{fig:3}(e)-(f).

Finally, let's consider the case where $0<p<1$ (the infection rate is not directly proportional to the number of infected individuals). In this scenario, both models \eqref{model-1} and \eqref{model-2} yield consistent results, as shown in \cite[Theorems 2.1(ii), 2.2(ii), and 2.3(ii)]{PSW} and Theorems \ref{TH2.8}(ii), \ref{TH2.9}(ii), and \ref{TH2.10}(ii). These results indicate that restricting population movement alone cannot eliminate the disease in any location. However, it remains an open problem to investigate whether strategies that limit population movement can effectively reduce the number of infected individuals in both models.


In conclusion, our main findings in \cite{PSW} and this paper suggest that when the infection rate is proportional to the number of infected individuals ($p=1$), different strategies for limiting population movement can lead to different outcomes in terms of disease elimination in the two models studied. However, when the infection rate is not proportional to the number of infected individuals ($0<p<1$), limiting population movement alone seems to be ineffective in completely eradicating the disease in both models. Interestingly, the power $q$ appears to have minimal influence on disease transmission in both scenarios. These theoretical results presented in \cite{PSW} and this paper may provide new insights for the development of effective and efficient measures for controlling infectious diseases.

\section{Appendix}\label{Appendix-section}

\subsection{Proof of Proposition \ref{theorem_ex}}

The proof of Proposition \ref{theorem_ex} is similar to that of model \eqref{model-1} in \cite{PW2021}, so we will provide a brief sketch of the proof here.

\begin{proof}[Proof of Proposition \ref{theorem_ex}] The first step is to show that if a nonnegative solution exists it satisfies \eqref{th-b1} and if the solution is global it satisfies  \eqref{th-b2}. To show these results, by adding up the two equations in
\eqref{model-2} and integrating it over $\Omega$, we can see that
\eqref{th-b1}-\eqref{th-b2} hold with the $L^\infty$ norm replaced by the $L^1$ norm.
If $p=1$, we can adopt the argument in  \cite[Theorem 2.1]{li2018diffusive} to prove the desired results. Indeed, consider the ODE problem:
\begin{equation}
u_t=\Lambda_{\max}+r^\frac{1}{q}_{\max}-u, \ \ t>0;\ \ \
u(0)=\max\{(S_0)_{\max},\ \Lambda_{\max},\ r^\frac{1}{q}_{\max}\},
\end{equation}
where $r=\gamma/\beta$. Then it is easy to see that $u(t)\ge r^\frac{1}{q}_{\max}$ for all $t\ge 0$. So $u$ satisfies
$$
u_t-d_S\Delta u\ge \Lambda -u-\beta u^qI+\gamma I.
$$
By the parabolic comparison principle, $S\le u$ for all $x\in\bar\Omega$ and $t\ge 0$. Then clearly, the desired results on the component $S$ hold. By using the equation of $I$ and \cite[Lemma 3.1]{PZ2012}, we can obtain the bounds on the component $I$.

If $0<p<1$, we can follow the techniques of \cite[Lemma 3.8]{PW2021}. For any $a>1$, let $b=(1-p)(a-1)/q+1$. By Young's inequality, one can show
\begin{equation}\label{ab}
S^{a-1}I\le \epsilon S^{q+a-1}I^p+C_\epsilon I^b \ \ \text{and} \ \ S^qI^{p+b-1}\le \epsilon S^{q+a-1}I^p+C_\epsilon I^b,\  \ \forall\, S,\, I>0,
\end{equation}
where $\epsilon>0$ can be arbitrarily small and $C_\epsilon>0$ depends on $\epsilon$. Multiplying the $S$ equation by $S^{a-1}$ and the $I$ equation by $I^{b-1}$, and then using \eqref{ab} and the interpolation inequality $\|w\|_{L^2(\Omega)}^2 \le \epsilon \|\nabla w\|_{L^2(\Omega)}^2+C_{\epsilon}\|w\|_{L^1(\Omega)}^2$ for all $w\in W^{1, 2}(\Omega)$, we obtain
\begin{eqnarray}\label{TTT1}
&&\frac{d}{dt}(\|S\|_{L^a(\Omega)}^a+\|I\|_{L^b(\Omega)}^b) \le\\
&&\quad \quad C_1\left(1+  \|S\|_{L^{a/2}(\Omega)}^a+\|I\|_{L^{b/2}(\Omega)}^b\right)-C_2(\|S\|_{L^a(\Omega)}^a+\|I\|_{L^b(\Omega)}^b),\ \forall t\ge 0, \nonumber
\end{eqnarray}
where $C_1, C_2$ depend on $a, b$ and the coefficients of the model but independent of initial data. By \eqref{TTT1} and induction, we can show that \eqref{th-b1}-\eqref{th-b2} hold with the $L^\infty$ norm replaced by the $L^k$ norm for any $k>1$. Then using a semigroup method (e.g., see \cite[Theorem 3.1]{PW2021}), we can obtain  \eqref{th-b1}-\eqref{th-b2}.

The second step is to show that a nonnegative solution, if exists, is positive for all $t>0$. Suppose that a nonnegative solution exists on $\bar\Omega\times [0, T]$. Let $M=\sup_{(x, t)\in\bar\Omega\times [0, T]} I(x, t)$. Then $S$ satisfies
$$
S_t-d_S\Delta S\ge \Lambda-S-\beta S^qM^p.
$$
By the maximum principle, one can show $S(x, t)>0$ for all $(x, t)\in\bar\Omega\times [0, T]$. Similarly, one can show the positivity of $I$ by using $I_t-d_I\Delta I\ge -(\gamma+\mu)I$.

Finally, by assumptions \rm(A1)-(A2), we can use the Banach fixed point theorem to show that \eqref{model-2} has a unique nonnegative solution on $\bar\Omega\times [0, T']$, where $T'>0$ depends only on the $L^\infty$-norm of the initial data. By the first two steps, the solution can be extended to be a positive global solution, and so the assertion (i) holds.

Based on the assertion (i), the proof of (ii) is similar to that of \cite[Theorem 2.5]{PW2021} and we omit the details here.
\end{proof}

\subsection{A weak convergence result of function sequences} In this subsection, we shall establish the following weak convergence result of nonnegative  function sequences, which will be used in the proof of Theorem \ref{TH2.8}.

\begin{prop}\label{lemma_fg} Let $q>0$ be a given constant. Suppose that $\{f_n\}$ are bounded nonnegative functions in $L^\infty(\Omega)$ and precompact in $L^2(\Omega)$, and $\{g_n\}$ are bounded nonnegative functions  in $L^1(\Omega)$. Then there exist subsequences $\{f_{n_k}\}$ and $\{g_{n_k}\}$ such that
\begin{eqnarray}\nonumber
f_{n_k}\overset{\ast}{\rightharpoonup} f\ \ \text{in}\ L^\infty(\Omega),\ \
g_{n_k}\overset{\ast}{\rightharpoonup} \nu\ \ \text{in}\ [C(\bar\Omega)]^*,\ \
\nonumber (f_{n_k})^{q}g_{n_k}\overset{\ast}{\rightharpoonup} \mu\ \ \text{in}\ [C(\bar\Omega)]^*,
\end{eqnarray}
where $f\in L^\infty(\Omega)$,  $\nu$ and $\mu$ are Radon measures. Here, $[C(\bar\Omega)]^*$  is the dual of $C(\bar\Omega)$. Moreover there exist measurable set $F_1$ with Lebesgue measure zero and closed sets $F_l$, $l\ge 2$, with $\Omega=\cup_{l\ge 1} F_l$ such that $f$ is continuous on $F_l$ for all $l\ge 2$, and $\mu=f^q\nu$ on $\cup_{l\ge 2}F_l$  in the sense that
    $$
    \int_O \varphi d\mu=\int_O f^q\varphi d\nu
    $$
for any Borel set $O\subset\cup_{l\ge 2}F_l$ and $\varphi\in C(\bar\Omega)$.
    \end{prop}

\begin{proof} Since $\{f_n\}$ are bounded in $L^\infty(\Omega)$ and $[L^1(\Omega)]^*=L^\infty(\Omega)$,
the Banach-Alaoglu theorem enables us to conclude that, there is a subsequence $f_{n_k}$ such that $f_{n_k}\overset{\ast}{\rightharpoonup} f$ in $L^\infty(\Omega)$ as $k\to\infty$.
Note that $\{(f_{n_k})^qg_{n_k}\}$  and $\{g_{n_k}\}$  are both bounded in $L^1(\Omega)$, by the Riesz representation theorem and the Banach-Alaoglu theorem, up to a further subsequence, we may assume that, as $k\to\infty$, $g_{n_k}\overset{\ast}{\rightharpoonup}\nu$ in $[C(\bar\Omega)]^*$, and $f_{n_k}^qg_{n_k}\overset{\ast}{\rightharpoonup} \mu$ in $[C(\bar\Omega)]^*$, where $\nu$ and $\mu$ are finite Radon measures.

Since $\{f_n\}$ is precompact in $L^2(\Omega)$, we may further assume $f_{n_k}$ converges pointwise to $f$ a.e in $\Omega$. By the Egoroff's theorem, there exists closed set $G_l$ with $\omega(G_l^c)<\frac{1}{2l}$ for any $l\ge 2$ such that $f_{n_k}\to f$ uniformly on $G_l$ as $k\to\infty$. Here, $\omega$ denotes the Lebesgue measure and $G_l^c$ is the complement of $G_l$ in $\Omega$. In addition, due to the Lusin's theorem, there exists closed set $H_l$ with $\omega(H_l^c)<\frac{1}{2l}$ for any $l\ge 2$ such that $f$ is continuous on $H_l$ as $k\to\infty$.

Set $F_l=G_l\cap H_l$ for $l\ge 2$, and $F_1=\cap_{l\ge 2} F_l^c$. Then $F_l$ is closed such that  $f_{n_k}$ converges uniformly to $f$ and $f$ is continuous on $F_l$ for all $l\ge 2$. In view of
    $$
    \omega(F_l^c)=\omega(G_l^c\cup H_l^c)\le \omega(G_l^c)+\omega(H_l^c)<\frac{1}{2l}+\frac{1}{2l}=\frac{1}{l},
    $$
we have $\omega(F_1)=\lim_{l\to\infty} \omega(F_l^c)=0$. Moreover, $\cup_{l\ge 1}F_l=\Omega$.

Let $\tilde F_1=F_1$ and $\tilde F_l=F_l\backslash\cup_{1\le j\le l-1}\tilde F_j$ for $j\ge 2$.
Then $\tilde F_l \ (l\ge 1)$ are disjoint Borel sets. For any $l\ge 2$ and $\varphi\in C(\bar\Omega)$, there holds that
    \begin{eqnarray*}
        &&\left|\int_{\tilde F_l} \varphi d\mu-\int_{\tilde F_l} \varphi f^qd\nu\right|\\
        &\le&    \left|\int_{\tilde F_l} \varphi d\mu-\int_{\tilde F_l} \varphi f_{n_k}^qg_{n_k}\right|+\left|\int_{\tilde F_l} \varphi f_{n_k}^qg_{n_k}-\int_{\tilde F_l} \varphi f^qg_{n_k}\right|+\left|\int_{\tilde F_l} \varphi f^qg_{n_k}-\int_{\tilde F_l} \varphi f^qd\nu\right|\\
    &\le& \left|\int_{\tilde F_l} \varphi d\mu-\int_{\tilde F_l} \varphi f_{n_k}^qg_{n_k}\right|+\|\varphi\|_{L^\infty(\Omega)}\|f_{n_k}^q-f^q\|_{L^\infty(\tilde F_l)}\int_{\tilde F_l}g_{n_k}
    \\
    &&\ \ +\left|\int_{\tilde F_l} \varphi f^qg_{n_k}-\int_{\tilde F_l} \varphi f^qd\nu\right|.
    \end{eqnarray*}
    Taking $k\to \infty$ and using $f_{n_k}\to f$ uniformly on $\tilde F_l$, we obtain
    $$
    \int_{\tilde F_l} \varphi d\mu-\int_{\tilde F_l} \varphi f^qd\nu=0, \ \ \forall l\ge 2.
    $$
For any Borel set $O\subset \cup_{l\ge 2}F_l= \cup_{l\ge 2}\tilde F_l$, we have
\begin{equation}\label{fd}
    \int_O \varphi d\mu=\sum_{l\ge 2}\int_{\tilde F_l}\varphi \chi_{O}d\mu=\sum_{l\ge 2}\int_{\tilde F_l} \varphi f^q\chi_{O}d\nu=\int_O f^q\varphi d\nu,
\end{equation}
    where $\chi_O$ is the characteristic function of $O$. This establishes Proposition \ref{lemma_fg}.
\end{proof}

\subsection{Analysis on a cooperative system}
In this subsection, we fix strictly positive  H\"older continuous functions $\tilde{\eta}$, $\tilde{\Lambda}$ and $\tilde{\gamma}$, and set $\tilde{h}=\frac{\tilde{\gamma}+\tilde{\eta}}{\beta}$ and $\tilde{r}=\frac{\tilde{\gamma}}{\beta}$. Then we can establish the following result.

\begin{prop}\label{prop-6.2} Let $q>0$ and  ${\sigma}>0$ such that ${\sigma}> \tilde{\eta}_{\max}$. Suppose  that $\hat{\Om}:=\{x\in\Om:\ \tilde{\Lambda}(x)>\tilde{h}^\frac{1}{q}(x) \}\ne\emptyset$. Then there exists $d_*>0$ such that if $d_S+|\frac{d_I}{d_S}-\sigma|<d_*$ the cooperative system \begin{equation}\label{C9}
    \begin{cases}
\displaystyle        d_S\Delta \tilde{u}+\tilde{\Lambda} -\tilde{u}+\left(1-\frac{d_S}{d_I}\tilde{\eta}\right)\tilde{v}=0,\ \ \ & x\in\Om,\cr
\displaystyle       d_S\Delta \tilde{v}+\frac{d_S\beta}{d_I}\left[(\tilde{u}-\tilde{v})^q-\tilde{h}\right]\tilde{v}=0, & x\in\Om,\cr
        \partial_{\nu}\tilde{u}=\partial_\nu\tilde{v}=0, & x\in\partial\Om,\cr
        0<\tilde{v}<\tilde{u}, & x\in\bar{\Om}
    \end{cases}
\end{equation}
has at least one classical solution. Furthermore, every solution $(\tilde{u},\tilde{v})$ of \eqref{C9}  satisfies
\begin{equation}\label{C10}
\lim_{d_S+|\frac{d_I}{d_S}-\sigma|\to 0}(\tilde{u},\tilde{v})=\Big(\min\{\tilde{\Lambda},\tilde{h}^{\frac{1}{q}}\}
+\frac{\title{\sigma}}{\tilde{\eta}}\big(\tilde{\Lambda}-\tilde{h}^{\frac{1}{q}}\big)_+,\ \frac{\title{\sigma}}{\tilde{\eta}}\big(\tilde{\Lambda}-\tilde{h}^{\frac{1}{q}}\big)_+\Big)\ \ \text{uniformly on} \ \bar\Omega.
\end{equation}
\end{prop}


Set $\xi_*:=\frac{\min\{\tilde{\Lambda}_{\min},\ \tilde{r}_{\min}^{\frac{1}{q}}\}}{2}$ and define
 \begin{equation}\label{R-rectangle-star}
 \mathfrak{R}_*:=\Big\{(u,v)\in\mathbb{R}_+^2:\ 0\le v\le u-\xi_*\Big\}.
 \end{equation}
For any solution $(\tilde{u},\tilde{v})$ of \eqref{C9}, let $S=\tilde u-\tilde v$ and $I=\frac{d_S}{d_I}\tilde v$. One can apply the similar analysis as in Lemma \ref{lemma_Sbound} to the system that $(S,I)$ satisfies to prove that $S\geq\min\{\tilde{\Lambda}_{\min},\ \tilde{r}_{\min}^{\frac{1}{q}}\}\ge 2\xi_*$. This implies that, if $(\tilde{u},\tilde{v})$ is a solution of \eqref{C9}, then $(\tilde{u},\tilde{v})\in [C^2(\Om)]^2\cap [C^1(\bar{\Om})]^2\cap C(\bar{\Om}:\mathfrak{R}_*)$.

Let $d_S>0$ and $d_I>0$ such that $\tilde{\eta}_{\max}<\frac{d_I}{d_S}$. Note that every solution of \eqref{C9} is a positive steady state  of  the following system of parabolic equations:
\begin{equation}\label{C9-1}
    \begin{cases}
\displaystyle        \tilde{u}_t=d_S\Delta \tilde{u}+\tilde{\Lambda} -\tilde{u}+\left(1-\frac{d_S}{d_I}\tilde{\eta}\right)\tilde{v},\ \ \ & x\in\Om,\ t>0,\cr
\displaystyle        \tilde{v}_t=d_S\Delta \tilde{v}+\frac{d_S\beta}{d_I}\left[(\tilde{u}-\tilde{v})_+^q-\tilde{h}\right]\tilde{v}, & x\in\Om,\ t>0,\cr
        \partial_{\nu}\tilde{u}=\partial_{\nu}\tilde{v}=0, & x\in\partial\Om,\ t>0.
    \end{cases}
\end{equation}
For any initial data $(\tilde{u}(\cdot,0),\tilde{v}(\cdot,0))\in C(\bar{\Om}:\mathfrak{R}_*)$, it follows from standard arguments of parabolic equations theory that system \eqref{C9-1} has a unique classical solution $(\tilde{u}(x,t),\tilde{v}(x,t))$.
Note  that even $q\in(0,1)$, if $(\tilde{u}(\cdot,0),\tilde{v}(\cdot,0))\in C(\bar{\Om}:\ \mathfrak{R}_*) $, the nonlinear term $(\tilde{u}-\tilde{v})_+^q$ is Lipschitz continuous near $(\tilde{u}(\cdot,0),\tilde{v}(\cdot,0))$, which  justifies the uniqueness of the classical solution.

We will prove Proposition \ref{prop-6.2} through a couple of lemmas. The following result  shows that the unique classical solution of \eqref{C9-1} stays within $C(\bar{\Om}:\mathfrak{R}_*)$ at all  time.

\begin{lem}\label{Appendix-lem1}
For any $(\tilde{u}(\cdot,0),\tilde{v}(\cdot,0))\in C(\bar{\Om}:\mathfrak{R}_*)$,  the unique classical solution $(\tilde{u}(\cdot,t),\tilde{v}(\cdot,t))$ of \eqref{C9-1} satisfies  $(\tilde{u}(\cdot,t),\tilde{v}(\cdot,t))\in C(\bar{\Om}:\mathfrak{R}_*)$ for all $t\ge 0$.

\end{lem}
\begin{proof}  Let $\tilde{z}(x,t)=\tilde{u}(x,t)-\tilde{v}(x,t)$ for $x\in\bar{\Om}$ and $t\ge 0$. Then $\tilde{z}$ satisfies
\begin{equation}\label{C9-1-aa}
\begin{cases}
\displaystyle \tilde{z}_t=d_S\Delta \tilde{z}+(\tilde{\Lambda} -\tilde{z})+\frac{d_S}{d_I}\beta(\tilde{r}-\tilde{z}_+^q)\tilde{v},\ \ \ & x\in\Om,\ t>0,\cr
\partial_{\nu}\tilde{z}=0, & x\in\partial\Om,\ t>0.
\end{cases}
\end{equation}
Noticing $\tilde{v}\ge 0$,  $\underline{z}\equiv\xi_*$ is a subsolution of \eqref{C9-1-aa}. Since $\tilde{z}(\cdot,0)\ge \xi_*$, it follows from the comparison principle for parabolic equations that $\tilde{z}(\cdot,t)\ge \underline{z}= \xi_*$ for all $t\ge 0$.
 \end{proof}

By Lemma \ref{Appendix-lem1}, the solution operator of \eqref{C9-1}  is invariant in $C(\bar{\Om}:\mathfrak{R}_*)$. Since $ 1>\frac{d_S}{d_I}\tilde{\eta}_{\max} $,  the solutions of the system \eqref{C9-1} generate a monotone semiflow on $C(\bar{\Om}:\mathfrak{R}_*)$.

\begin{lem}\label{Appendix-lem2}
Every classical solution of system \eqref{C9-1} with initial data in $C(\bar{\Om}:\mathfrak{R}_*)$ is uniformly bounded.
\end{lem}
\begin{proof} By elementary computations, we see that for every $m\ge 1$ the constant function
\begin{equation}\label{C9-2-ab}
(u_m,v_m)=\Big(m\tilde{h}_{\min}^{\frac{1}{q}}+m\frac{d_I}{d_S}
\Big(\frac{\tilde{\Lambda}}{\tilde{\eta}}\Big)_{\max},\ \, m\frac{d_I}{d_S}
\Big(\frac{\tilde{\Lambda}}{\tilde{\eta}}\Big)_{\max}\Big)
\end{equation}
is a supersolution of \eqref{C9-1}. Hence if $(\tilde{u}(\cdot,0),\tilde{v}(\cdot,0))\le (u_m,v_m)$  for some $m\ge 1$, then $(\tilde{u}(\cdot,t),\tilde{v}(\cdot,t))\le (u_{m},v_{m})$ for all $t>0$.
\end{proof}

In view of Lemma \ref{Appendix-lem2}, the monotone  semiflow generated by classical solutions of the system \eqref{C9-1} on $C(\bar{\Om}:\ \mathfrak{R}_*)$ is compact.

\begin{lem}\label{Appendix-lem3} Suppose that $\hat{\Om}:=\{x\in\Om:\ \tilde{\Lambda}(x)>\tilde{h}^\frac{1}{q}(x) \}\ne\emptyset$.  Then there is $d_{*}>0$ such that \eqref{C9-1} has at least one positive steady state in $C(\bar{\Om}:\mathfrak{R}_*)$  if $d_S+|\frac{d_I}{d_S}-\sigma|<d_*$.
\end{lem}

\begin{proof} Let $\tilde{S}$ be the unique (positive) solution of
\begin{equation}\label{au-1}
    \begin{cases}
        d_S\Delta u-u+\tilde{\Lambda}=0,\ \ \ & x\in\Om,\cr
        \partial_{\nu}u=0, &x\in\partial\Om.
    \end{cases}
\end{equation}
Note that $(\tilde{S},0)$ is a steady state solution of \eqref{C9-1} and belongs to $C(\bar{\Om}:\mathfrak{R}_*)$. Linearizing \eqref{C9-1} at $({\tilde{S}},0)$, we obtain the eigenvalue problem
\begin{equation}\label{C9-2}
\begin{cases}
\displaystyle d_S\Delta \varphi_1-\varphi_1+\big(1-\frac{d_S}{d_I}\eta\big)\varphi_2=\lambda\varphi_1,\ \ \ & x\in\Om,\cr
\displaystyle d_S\Delta\varphi_2+\frac{d_S\beta}{d_I}\big({\tilde{S}}^{q}-\tilde{h}\big)\varphi_2=\lambda\varphi_2, & x\in\Om,\cr
\partial_{\nu}\varphi_1=\partial_{\nu}\varphi_2=0, & x\in\partial\Om.
\end{cases}
\end{equation}
Clearly, the maximal eigenvalue of \eqref{C9-2} is given by
     $$
     \tilde{\Gamma}_{d_{S},d_{I}}:=\max\{-1,\ \,\lambda_{1,d_S,d_I}\},
     $$
where $\lambda_{1,d_S,d_I}$ is the principal eigenvalue of the following eigenvalue problem:
\begin{equation}\nonumber
\begin{cases}
\displaystyle d_S\Delta\varphi+\frac{d_S\beta}{d_I}\big({\tilde{S}}^{q}-\tilde{h}\big)\varphi=\lambda\varphi, & x\in\Om,\cr
\partial_{\nu}\varphi=0, & x\in\partial\Om.
\end{cases}
\end{equation}
Since $\tilde{S}\to \tilde{\Lambda}$ uniformly on $\bar\Omega$ as $d_S\to0$ and $\hat{\Om}\ne\emptyset$, it follows from \cite[Lemma 3.1]{LN2006}  and the continuity of the principal eigenvalue with respect to parameters that
\begin{equation}
\lim_{d_S+|\frac{d_I}{d_S}-\sigma|\to 0}\tilde{\Gamma}_{d_{S},d_{I}}=\max\left\{-1,\ \,\max_{x\in\bar{\Om}}\frac{\beta(x)}{{\sigma}}\big(\tilde{\Lambda}^q(x)-\tilde{h}(x)\big)\right\}>0.
\end{equation}
As a consequence, there is $d_*>0$ such that $\tilde{\Gamma}_{d_{S},d_{I}}>0$ and $({\tilde{S}},0)$ is linearly unstable if $d_S+|\frac{d_I}{d_S}-\sigma|<d_{*}$.

Fix $d_S>0$ and $d_I>0$ such that $d_S+|\frac{d_I}{d_S}-\sigma|<d_*$. Let $\tilde{\varphi}_2$ be the positive eigenfunction associated with $\lambda_{1,d_S,d_I}$ satisfying $\max_{x\in\bar{\Omega}}\tilde{\varphi}_2(x)=1$. Since $\tilde{\Gamma}_{d_S,d_I}=\lambda_{1,d_S,d_I}>0$, then $\lambda_{1,d_S,d_I}+1$ belongs to the resolvent set of $d_S\Delta$ subject to  homogeneous Neumann boundary condition. Hence, there is a unique $\tilde{\varphi}_1$ solving the first equation of \eqref{C9-2}  with $\lambda=\lambda_{1,d_I,d_I}$ and $\varphi_2=\tilde{\varphi}_2$.
Since $\tilde{\varphi}_2>0$, we obtain from the strong maximum principle for elliptic equations that $\tilde{\varphi}_1>0$ on $\bar{\Omega}$.
Define  $(\underline{u}_{\varepsilon},\underline{v}_{\varepsilon}):=(\tilde{S}+\varepsilon\tilde{\varphi}_1,\varepsilon\tilde{\varphi}_2)$ for $\varepsilon>0$. Since $\tilde{S}_{\min}\ge \tilde{\Lambda}_{\min}$,  we can fix $0<\varepsilon\ll 1$ such that $ (\underline{u}_{\varepsilon},\underline{v}_{\varepsilon})\in C(\bar{\Omega}: \mathcal{R}_*)$.
Moreover, we can choose $0<\varepsilon\ll 1$ such that
\begin{equation*}
      \begin{cases}  d_S\Delta \displaystyle\underline{u}_{\varepsilon}+\tilde{\Lambda}-\underline{u}_{\varepsilon}+(1-\frac{d_S}{d_I}\tilde{\eta})\underline{v}_{\varepsilon}=\varepsilon\lambda_{1,d_S,d_I}\tilde{\varphi}_1>0, & x\in\Omega,\cr
      \partial_{\nu}\underline{u}_{\varepsilon}=0, & x\in\partial\Omega
      \end{cases}
\end{equation*}
and
\begin{align*}
\begin{cases}
   \displaystyle d_S\Delta\underline{v}_{\varepsilon}+\frac{d_S}{d_I}\beta\big[(\underline{u}_{\varepsilon}-\underline{v}_{\varepsilon})^q-\tilde{h}\big]\underline{v}_{\varepsilon}= \\
   \quad\quad\quad\quad\left(\lambda_{1,d_S,d_I}+\frac{d_S}{d_I}\beta \left[(\tilde S+\varepsilon (\tilde\varphi_1-\tilde\varphi_2))^q-\tilde S^q)\right]\right)\underline{v}_{\varepsilon}>0, & x\in\Omega, \cr
      \partial_{\nu}\underline{v}_{\varepsilon}=0, & x\in\partial\Omega.
    \end{cases}
\end{align*}
Therefore, $(\underline{u}_{\varepsilon},\underline{v}_{\varepsilon})$ is a subsolution of \eqref{C9-1}.     Also recall that $(u_m,v_m)$ given by \eqref{C9-2-ab} is a supersolution of \eqref{C9-1} for every $m\ge 1$. Thus, choosing $m\gg 1$ such that $(\underline{u}_{\varepsilon},\underline{v}_{\varepsilon})\le  (u_m,v_m)$,  the standard arguments from the monotone dynamical systems and the precompactness of every orbit in $C(\bar{\Om}:\mathfrak{R}_*)$ ensure that \eqref{C9-1} has at least one positive steady state in  $[(\underline{u}_{\varepsilon},\underline{v}_{\varepsilon}),(u_m,v_m)]\subset C(\bar{\Om}:\mathfrak{R}_*)$.
  \end{proof}

  \begin{lem} \label{Appendix-lem5}
  \begin{itemize}
      \item[\rm (i)] Let $\underline{u}_0\equiv \tilde{\Lambda}$ and $\underline{v}_0\equiv0$. Define the sequence $\{(\underline{u}_n,\underline{v}_n)\}$  inductively by $\underline{u}_n=\tilde{\Lambda}+(1-\frac{\tilde{\eta}}{\sigma})\underline{v}_{n}$ and $\underline{v}_{n+1}=(\underline{u}_n-\tilde{h}^{\frac{1}{q}})_+$ for every $n\ge 0$. Then the sequence $\{(\underline{u}_n,\underline{v}_n)\}$ is increasing and
      $$
      \lim_{n\to\infty}(\underline{u}_n,\underline{v}_n)
      =\Big(\min\{\tilde{\Lambda},\tilde{h}^{\frac{1}{q}}\}
      +\frac{\sigma}{\tilde{\eta}}(\tilde{\Lambda}-\tilde{h}^{\frac{1}{q}})_+,\ \frac{\sigma}{\tilde{\eta}}(\tilde{\Lambda}-\tilde{h}^{\frac{1}{q}})_+\Big)\quad  \text{uniformly on}\ \bar\Omega.
      $$
      \item[\rm (ii)] Let $\overline{v}_0=\overline{u}_0=\tilde{\Lambda}_{\max}+(\sigma+1)
          \big(\frac{\tilde{\Lambda}}{\tilde{\eta}}\big)_{\max}$. Define the sequence $\{(\overline{u}_n,\overline{v}_n)\}$ inductively by $\overline{u}_{n+1}=\Lambda+(1-\frac{\tilde\eta}{\sigma})\overline{v}_n$ and $\overline{v}_{n+1}=(\overline{u}_{n}-{\tilde h}^{\frac{1}{q}})_+$ for every $n\ge 0$. Then $\{(\overline{u}_n,\overline{v}_n)\}$ is decreasing and
      $$
      \lim_{n\to \infty}(\overline{u}_n,\overline{v}_n)=\Big(\min\{\tilde{\Lambda},\tilde{h}^{\frac{1}{q}}\}
      +\frac{\sigma}{\tilde{\eta}}(\tilde{\Lambda}-\tilde{h}^{\frac{1}{q}})_+,\
      \frac{\sigma}{\tilde{\eta}}(\tilde{\Lambda}-\tilde{h}^{\frac{1}{q}})_+\Big)\quad  \text{uniformly on}\ \bar\Omega.
      $$
  \end{itemize}

  \end{lem}
\begin{proof}{\rm (i)} We proceed by induction to show the monotonicity of  $\{(\underline{u}_n,\underline{v}_n)\}$. It is clear that $\underline{u}_1=\tilde{\Lambda}\ge\underline{u}_0$ and $\underline{v}_1=(\tilde{\Lambda}-\tilde{h}^{\frac{1}{q}})_+\ge \underline{v}_0$. So, we  suppose that the monotonicity holds up to some $n$ and  check that it also holds for $n+1$. Indeed, by the induction hypothesis,
$$
\underline{v}_{n+1}=(\underline{u}_{n}-\tilde{h}^{\frac{1}{q}})_+\ge (\underline{u}_{n-1}-\tilde{h}^{\frac{1}{q}})_+=\underline{v}_n,
$$
and in turn
$$
\underline{u}_{n+1}=\tilde{\Lambda}+(1-\frac{\tilde{\eta}}{\sigma})\underline{v}_{n+1}\ge \tilde{\Lambda}+(1-\frac{\tilde{\eta}}{\sigma})\underline{v}_{n}=\underline{u}_n.
$$
This verifies the monotonicity of the sequence $\{(\underline{u}_n,\underline{v}_n)\}$.

Next, using
\begin{equation}\label{C9-4}
\underline{v}_{n+1}=\left[\tilde{\Lambda}+\left(1-\frac{\tilde{\eta}}{\sigma}\right)
\underline{v}_n-\tilde{h}^{\frac{1}{q}}\right]_+,\quad \forall\;n\ge 0,
\end{equation}
we deduce that
$$
\underline{v}_{n+1}\le \left(1-\frac{\tilde{\eta}}{\sigma}\right)\underline{v}_{n}+(\tilde{\Lambda}-\tilde{h}^{\frac{1}{q}})_+, \quad \forall \;n\ge 0.
$$
Therefore,
\begin{equation}\label{C9-5}
\underline{v}_{n+1}\le (\tilde{\Lambda}-{\tilde h}^{\frac{1}{q}})_+\sum_{i=0}^{n}\Big(1-\frac{\tilde{\eta}}{\sigma}\Big)^i,\quad n\ge 1.
\end{equation}
Thanks to $\underline{v}_1=(\tilde{\Lambda}-\tilde{h}^{\frac{1}{q}})_+$, we obtain from the monotonicity of $\{\underline{v}_n\}_{n\ge 0}$ and the inequality \eqref{C9-5} that $\underline{v}_{n}(x)>0$   if and only if $x\in\hat{\Om}$ for all $n\ge 1$. Taking $n\to\infty$ in \eqref{C9-4} yields $\underline{v}_n\to \frac{\sigma}{\tilde{\eta}}(\tilde{\Lambda}-\tilde{h}^{\frac{1}{q}})_+$ uniformly on $\bar\Omega$. Since $\underline{u}_n=\tilde{\Lambda}+(1-\frac{\tilde{\eta}}{\sigma})\underline{v}_n$ for all $n\ge 0$, we deduce that
$$\underline{u}_n\to \tilde{\Lambda}+\big(\frac{\sigma}{\tilde{\eta}}-1\big)(\tilde{\Lambda}-\tilde{h}^{\frac{1}{q}})_+
=\min\{\tilde{\Lambda},\tilde{h}^{\frac{1}{q}}\}
+\frac{\sigma}{\tilde{\eta}}(\tilde{\Lambda}-\tilde{h}^{\frac{1}{q}})_+
$$
uniformly on $\bar\Om$  as $n\to\infty$.

\quad {\rm (ii)} We proceed by induction to show that $\{(\overline{u}_n,\overline{v}_n)\}$ is decreasing. Here it  suffices to show that $\overline{u}_1\le \overline{u}_0$ and $\overline{v}_1\le\overline{v}_0$. Clearly, it holds that
$$
\overline{u}_0-\overline{u}_1=\overline{u}_0-\tilde{\Lambda}-\big(1-\frac{\tilde{\eta}}{\sigma}\big)\overline{u}_0
=\frac{\tilde{\eta}}{\sigma}\overline{u}_0-\Lambda>0\quad \text{and}\quad \overline{v}_1=(\overline{v}_0-\tilde{h}^{\frac{1}{q}})_+\le \overline{v}_0.
$$
Since $\{(\overline{u}_n,\overline{v}_n)\}$ is decreasing and bounded below,  it converges to some $(\overline{u}_{\infty},\overline{v}_{\infty})$ pointwise on $\bar\Om$. Observe from the expressions of $\overline{u}_{n+1}$ and $\overline{v}_{n+1}$ that $(\overline{u}_{\infty},\overline{v}_{\infty})$ satisfies
$$
\overline{u}_{\infty}=\tilde{\Lambda}+\big(1-\frac{\tilde{\eta}}{\sigma}\big)\overline{v}_{\infty},\ \ \
\overline{v}_{\infty}=(\overline{u}_\infty-\tilde{h}^{\frac{1}{q}})_+.
$$
Solving this system yields
$$
(\overline{u}_{\infty},\overline{v}_{\infty})
=\left(\min\{\tilde{\Lambda},\ \tilde{h}^{\frac{1}{q}}\}+\frac{\sigma}{\tilde{\eta}}(\tilde{\Lambda}-\tilde{h}^{\frac{1}{q}})_+,\ \frac{\sigma}{\tilde{\eta}}(\tilde{\Lambda}-\tilde{h}^{\frac{1}{q}})_+\right).
$$
Since $(\overline{u}_{\infty},\overline{v}_{\infty})$ is continuous on $\bar\Omega$, we can invoke the Dini's Theorem to conclude that the convergence of $\{(\overline{u}_n,\overline{v}_n)\}$  is uniform on $\bar{\Om}$.
\end{proof}

Now, we can  obtain the asymptotic profiles of the solutions of \eqref{C9} as $d_I+|\frac{d_I}{d_S}-\sigma|\to0$:

\begin{lem}\label{Appendix-lem6} Let $d_{*}>0$ be given by Lemma \ref{Appendix-lem3}. Then every solution $(\tilde{u},\tilde{v})\in C(\bar{\Om}:\mathfrak{R}_*)$ of \eqref{C9} for $d_S+|\frac{d_I}{d_S}-\sigma|<d_{*}$ satisfies \eqref{C10}.

\end{lem}
\begin{proof} Let $(\tilde{u},\tilde{v})\in C(\bar{\Om}:\mathfrak{R}_*)$ be a solution of \eqref{C9} for $d_S+|\frac{d_I}{d_S}-\sigma|<d_{*}$, which exists by Lemma \ref{Appendix-lem3}.  Let $\{(\underline{u}_n,\underline{v}_n)\}$ be the sequence in Lemma \ref{Appendix-lem5}-{\rm(i)}. We claim that
\begin{equation}\label{C9-8}
    \liminf_{d_S+|\frac{d_I}{d_S}-\sigma|\to0}(\tilde{u},\tilde{v})\ge (\underline{u}_n,\underline{v}_n)\ \ \quad \text{uniformly on}\ \bar\Om, \ \forall\, n\ge 0.
\end{equation}

We proceed by induction to establish \eqref{C9-8}. Since $\tilde{v}\ge 0$, it follows from the comparison principle for elliptic equations that $\tilde{u}\ge {\tilde{S}}$, where ${\tilde{S}}$ is given by \eqref{au-1}. Since ${\tilde{S}}\to\tilde{\Lambda}=\underline{u}_0$ uniformly on $\bar\Omega$ as $d_S\to0$, we have
$$
\liminf_{d_{S}+|\frac{d_I}{d_S}-\sigma|\to0}\tilde{u}(x)\ge\underline{u}_0\quad\ \  \text{uniformly on}\ \bar\Om.
$$
It clear that
$$
\liminf_{d_{S}+|\frac{d_I}{d_S}-\sigma|\to0}\tilde{v}(x)\ge0=\underline{v}_0\quad\ \  \text{uniformly on}\ \bar\Om.
$$
Hence \eqref{C9-8} holds for $n=0$.

Suppose that \eqref{C9-8} holds up to some $n\ge 0$. Then we show that it also holds for $n+1$. Let $0<\varepsilon\ll 1$ be given. By the induction hypothesis, there is $d_{\varepsilon,1}>0$ such that $\tilde{u}\ge \underline{u}_n-\varepsilon $ on $\bar{\Om}$ if $d_S+|\frac{d_I}{d_S}-\sigma|<d_{\varepsilon,1}$. Thus,
$$
\begin{cases}
\displaystyle d_S\Delta \tilde{v}+\frac{d_S}{d_I}\beta\big[(\underline{u}_n-\varepsilon-\tilde{v})_+^q-\tilde{h}\big]\tilde{v}\leq0,\ \ \ & x\in\Om,\cr
\partial_{\nu}\tilde{v}=0, & x\in\partial\Om,
\end{cases}
$$
for all $d_S+|\frac{d_I}{d_S}-\sigma|<d_{\varepsilon,1}$. It then follows from the singular perturbation theory and the comparison principle  for elliptic equations that
$$
\liminf_{d_S+|\frac{d_I}{d_S}-\sigma|\to 0}\tilde{v}\ge (\underline{u}_n-\varepsilon-\tilde{h}^{\frac{1}{q}})_+\ \ \quad \text{uniformly on}\ \bar\Om.
$$
Since $\varepsilon$ is arbitrarily chosen, we have that
$$
\liminf_{d_S+|\frac{d_I}{d_S}-\sigma|\to 0}\tilde{v}\ge (\underline{u}_n-h^{\frac{1}{q}})_+=\underline{v}_{n+1}\ \ \quad \text{uniformly on}\ \bar\Om.
$$
This shows that there is $d_{\varepsilon,2}>0$ such that
$\tilde{v}\ge (\underline{v}_{n+1}-\varepsilon)_+$ on $\bar\Om$ for all $d_S+|\frac{d_I}{d_S}-\sigma|<d_{\varepsilon,2}$. As a result,
$$
\begin{cases}
\displaystyle d_S\Delta \tilde{u}-\tilde{u}+\tilde{\Lambda}+\big(1-\frac{d_I\tilde{\eta}}{d_S}\big)(\underline{v}_{n+1}-\varepsilon)_+\leq0,\ \ \ & x\in\Om\cr
\partial_{\nu}\tilde{u}=0, & x\in\partial\Om,
\end{cases}
$$
for all $d_S+|\frac{d_I}{d_S}-\sigma|<d_{\varepsilon,2}$. Again, we can employ the singular perturbation theory and the comparison principle for elliptic equation to obtain that
$$
\liminf_{d_S+|\frac{d_I}{d_S}-\sigma|\to0}\tilde{u}\ge \tilde{\Lambda}+\big(1-\frac{\tilde{\eta}}{\sigma}\big)(\underline{v}_{n+1}-\varepsilon)_+\quad\ \ \text{uniformly on}\ \bar\Om.
$$
This obviously implies that
$$
\liminf_{d_S+|\frac{d_I}{d_S}-\sigma|\to0}\tilde{u}\ge \tilde{\Lambda}+\big(1-\frac{\tilde{\eta}}{\sigma}\big)\underline{v}_{n+1}=\underline{u}_{n+1}\ \ \quad \text{uniformly on}\ \bar\Om.
$$
Therefore, \eqref{C9-8} holds for $n+1$. So by induction,  \eqref{C9-8} holds for all $n\ge 0$.

Let $\{(\overline{u}_n,\overline{v}_n)\}_{n\ge 0}$ be the sequence in Lemma \ref{Appendix-lem5}-{\rm(ii)}. We claim that, for each $n\ge 0$,
\begin{equation}\label{C9-9}
    \limsup_{d_S+|\frac{d_I}{d_S}-\sigma|\to0}(\tilde{u},\tilde{v})\le (\overline{u}_n,\overline{v}_n)\ \ \quad \text{uniformly on}\ \bar\Om.
\end{equation}

The proof of \eqref{C9-9} follows from an induction argument similar as above. Note that the argument leading to \eqref{C3} implies that both $\tilde{u}$ and $\tilde{v}$ also satisfy \eqref{C3} with $(\Lambda,h,\eta) $ replaced by $(\tilde{\Lambda},\tilde{h},\tilde{\eta})$.  It  then follows that $(\tilde{u},\tilde{v})\le (\overline{u}_0,\overline{v}_0)$ on $\bar\Om$. Hence, \eqref{C9-9} holds for $n=0$.

Suppose that \eqref{C9-9} holds up to some $n\ge 0$. Let $\varepsilon>0$ be fixed. By the induction hypothesis, there is $d_{\varepsilon}>0$ such that
$$
\tilde{u}\le \overline{u}_n+\varepsilon \quad\text{and}\quad \tilde{v}\le \overline{v}_n+\varepsilon, \quad \forall\; d_S+|\frac{d_I}{d_S}-\sigma|<d_{\varepsilon}.
$$
As a consequence, we have
\begin{equation}\label{C9-10}
    \begin{cases}
\displaystyle        d_S\Delta \tilde{u}+\tilde\Lambda-\tilde{u}+\big(1-\frac{d_I\tilde{\eta}}{d_S}\big)(\overline{v}_n+\varepsilon)\geq0,\ \ \ & x\in\Om,\cr
        \partial_{\nu}\tilde{u}=0, & x\in\Om
    \end{cases}
\end{equation}
and
\begin{equation}\label{C9-11}
\begin{cases}
\displaystyle d_S\Delta\tilde{v}+\frac{d_S}{d_I}\beta\Big[(\overline{u}_{n}+\varepsilon-\tilde{v})^q-\tilde{h}\Big]\tilde{v}\geq0,\ \ \ &x\in\Om,\cr
\partial_{\nu}\tilde{v}=0, & x\in\partial\Om,
\end{cases}
\end{equation}
if $d_S+|\frac{d_I}{d_S}-\sigma|<d_{\varepsilon}$. By the singular perturbation theory and the comparison principle for elliptic equations, we deduce from \eqref{C9-10} and \eqref{C9-11} that
$$
\limsup_{d_S+|\frac{d_I}{d_S}-\sigma|\to0}\tilde{u}\le \tilde{\Lambda}+\big(1-\frac{\tilde{\eta}}{\sigma}\big)(\overline{v}_n+\varepsilon)
=\overline{u}_{n+1}+\big(1-\frac{\tilde{\eta}}{\sigma}\big)\varepsilon
$$
and
$$
\limsup_{d_S+|\frac{d_I}{d_S}-\sigma|\to0}\tilde{v}\le (\overline{u}_n+\varepsilon-\tilde{h}^{\frac{1}{q}})_+\le \overline{v}_{n+1}+\varepsilon
$$
uniformly on $\bar\Om$. Since $\varepsilon>0$ was arbitrary,  \eqref{C9-9} holds for $n+1$. So, \eqref{C9-9} holds for all $n\ge 0$.

Taking $n\to\infty$ in \eqref{C9-8} and \eqref{C9-9}, we can see from Lemma \ref{Appendix-lem5} that $(\tilde{u},\tilde{v})$ satisfies \eqref{C10}.
\end{proof}

\vskip10pt
{ 
\noindent\textbf{Acknowledgment}. The authors would like to thank the referees for value suggestions that lead to improvement of the manuscript. }

\vskip25pt
\begin{center}
DATA AVAILABILITY STATEMENT
\end{center}
Data sharing is not applicable to this article as no new data were created or analysed in this study.

\vskip25pt

\bibliographystyle{plain}
%


\end{document}